\documentclass[11pt]{article}
\usepackage{fullpage}
\usepackage{amsmath,amssymb, amsthm}
\usepackage{ulem}
\usepackage{cite}
\usepackage{multicol}
\usepackage{relsize}
\usepackage{authblk}
\usepackage{color}

\usepackage{url}

\DeclareMathOperator{\Div}{div}	% divergence
\newcommand{\norm}[1]{\left\|#1\right\|}   %	norm
\newcommand{\abs}[1]{\left\lvert#1\right\rvert}  %	absolute value
\newcommand{\field}[1]{\mathbb{#1}}

\newcommand{\vektor}[1]{{\mathbf{#1}}} % vector
\newcommand{\vek}[1]{\vec{#1}} % vector for Y
\def\ro{\varrho}  %  density
\def\en{\vektor{n}}  % normal
\def\ef{\vektor{f}}
\def\sumk{\sum_{k=1}^{n}}
\newcommand{\R}{\mathbb{R}}

\def\u{\vektor{u}} % velocity

\newcommand{\de}[1]{\mathrm{d}{#1}}     %differential  for integrals
\def\dx{\de{x}}	
\newcommand{\inte}[1]{\int\limits_{\Omega}{#1}\, \dx  }
\newcommand{\inth}[1]{\int_{\partial\Omega}{#1}\, \de{S}  }
\newcommand{\refx}[1]{(\ref{#1})}
\def\qqquad{\quad\quad\quad} % spacing feature

\theoremstyle{plain}
\newtheorem{thm}{Theorem}%[chapter]
\newtheorem{lem}{Lemma}

\newtheorem{remark}{Remark}

\title{Steady solutions to a model of compressible chemically reacting fluid with high density}

\author[1]{\v{S}imon Axmann}

\author[2]{Milan Pokorn\'{y}\small}

\affil[1]{Department of Mathematics, University of Chemistry and Technology, 166 28 Praha 6, Czech Republic}

\affil[2]{Charles University, Faculty of Mathematics and Physics, Sokolovsk\'{a} 83, 186 75 Praha 8, Czech Republic\\
e-mail: {\tt pokorny@karlin.mff.cuni.cz} }

\begin{document}

\maketitle

\medskip
{\bf Acknowledgement:} The work of M.P. was supported by the grant of the Czech Science
Foundation No. 19-04243S.

\begin{abstract}
We consider a model describing the steady flow of compressible heat-conducting chemically-reacting multi-component mixture.
We show the existence of strong solutions  under the additional assumption that the mixture is sufficiently dense. We work in the $L^p$-setting combining the methods for the weak solutions with the method of decomposition. The result is a generalization of \cite{AMP16} and \cite{AMP18}, where the case of single-constituted fluid was studied.
\end{abstract}


% % % % % % % % % % % % % % % % % % % % % % %

\section{Introduction}

% % % % % % % % % % % % % % % % % % % % % %
 % % % % KNOWN RESULTS % % % %
% % % % % % % % % % % % % % % % % % % % % %

Recently, in mathematical fluid mechanics, it is possible to track increasing interest in models dealing with mixtures of different kinds. This is connected, indeed, with applications, since these models can describe processes which are important in many different areas, in chemistry, food industry, in pharmacology etc.

In addition, within the last decades, much better understanding of simpler models of compressible single-constituted fluids from the point of view of mathematical analysis allowed to deal also with much more complicated models which contain additionally to fluid mechanics (or thermodynamics) equations describing compressible fluid flow also further terms and equations which make possible to describe the behaviour of mixtures, even in the case when the components undergo chemical reactions.

In many applications the properties of weak solutions are too weak to be applicable, e.g. in the convergence analysis for numerical schemes. On the other hand, classical solutions typically exist only in situations when either the data are very small (so small that from the point of view of applications such results are often not of much interest) or, in case of evolutionary problem, on very short time intervals. Again, the latter usually means intervals of length $10^{-10}$ s, which, again is usually too short to have any reasonable application.

In this paper we want to study a problem which is somewhere in the middle of the results for weak and classical solutions. Based on previous results in this directions, obtained for steady isentropic compressible flow as well as steady flow of heat-conducting fluid flow of single constituted gas, we aim to extend these results also for a certain system of partial differential equations describing the steady flow of chemically reacting mixture of heat-conducting gases. The main point is that we obtain strong solutions to our system, i.e. solutions which solve the corresponding system of partial differential equations a.e., under the assumptions that the prescribed total mass of the mixture is sufficiently large. Indeed, in a sense this condition is similar to small data results, however, our assumptions allow to obtain solutions which are not just small perturbations of the flow with zero velocity and constant density and temperature.

Indeed, there is a price to pay for this. We have to restrict ourselves to a very specific assumptions on the model with the hope that we may later on extend these or similar results to more complex models.

\section{Formulation of the problem, known result}

% % % % % % % % % % % % % % % % % % % %
    % % % % SETTING % % % %
% % % % % % % % % % % % % % % % % % % %

Our model is based on the assumption that we take only the barycentric velocity into account and do not study separate velocities for each constituent. Therefore only one equation (or three as we work in the three--dimensional space) for the linear momentum appears. Similarly, we assume that the mixture is in the thermodynamic equilibrium in the sense that only one temperature and one specific internal energy is considered, even though these quantities may be different at different places. Therefore we deal with only one equation for the total energy of the whole mixture. To summarize, we consider the following steady version of a model of chemically reacting mixture with $n$ constituents in a bounded smooth domain $\Omega\subset \field{R}^3$
\begin{align}
     \Div(\ro \u) = &\:  0, \label{CEs} \\
 \Div(\ro \u\otimes\u) - \Div \field S  +   \nabla  \pi(\ro,\theta)=&\: \ro\mathbf{f},   \label{MEs} \\
 \Div(\ro e \u) + \Div \mathbf{Q}  =&\: \field S :\field D(\u)- \pi(\ro,\theta)\Div\u ,\label{EEs}\\
 \Div(\ro Y_k \u) + \Div \mathbf{F}_k  =&\: \ro m_k\omega_k, \quad k=1,\ldots,n. \label{Yks}
\end{align}
The unknowns are: density $\ro$: $\Omega\to (0,\infty)$, velocity field $\u$: $\Omega \to \field{R}^3$, absolute temperature $\theta$: $\Omega \to (0,\infty)$, and $Y_k$: $\Omega \to [ 0,1 ]$, $k=1,2,\dots, n$; mass fraction $Y_k$ corresponds to the $k$-th constituent; it is related to the density of the $k$-th constituent by $\ro_k = \ro Y_k$. Thus, since $\sumk \ro_k = \ro$, we demand $\sumk{Y_k} = 1$.

Above and in what follows, we use the following notation. Scalar-valued quantities (as e.g. the density $\ro$) are printed with the standard font. The vector-valued quantities $\Omega \to \field{R}^3$ (as e.g. the velocity $\u$) are printed in bold while the vector valued quantities $\Omega \to \field{R}^n$ (as e.g. $\vec Y = (Y_1,Y_2,\dots, Y_n)$) are printed with the arrow. Finally, the tensor-valued functions (as e.g. the stress tensor $\field{S}$) are printed with a special font.

The specific external force $\ef$: $\Omega \to \field{R}^3$ is given, as well as
the molar production rates $\omega_k = \omega_k(\vek{Y},\theta)$: $[0,1]^n \times (0,\infty)\to \field{R}$, $k=1,2,\dots, n$, are assumed to be given bounded functions such that for any admissible choice of $\vek{Y}$ and $\theta$ they satisfy
\begin{equation}\sumk m_k\omega_k=0.
\label{sumomeg}
 \end{equation}
This assumption means that no mass is created during the chemical reactions and the system can stay in equilibrium.
 Unfortunately, it is necessary for our approach to take all molar masses $m_k$ equal, for simplicity we choose $m_k=1$.
  In order to guarantee the non-negativity of the entropy production we also need to assume that (pointwisely, in the domain of the functions)
   \begin{equation}
\sumk \omega_k(\vek{Y},\theta) g_k(\ro,\vek{Y},\theta)\leq 0.\label{omega}
\end{equation}

The rest of the functions in system \refx{CEs}--\refx{Yks} will be also given functions of the unknowns through the following constitutive relations.

The viscous part of the stress tensor $\field S$ is a linear function with respect to the symmetric gradient of the velocity with density dependent viscosity, for the sake of simplicity we use
 \begin{equation}
\field S(\ro,\nabla\u) = 2\ro \field D(\u), \qquad \field D(\u) = \frac12 \bigl(\nabla \u + \nabla \u^T\bigr). \label{S}
\end{equation}
The fact that the viscosity depends on the density is in fact physically relevant. Moreover, as we may easily include also a nonzero bulk viscosity with similar form of the viscosity (it is only important that both viscosities are bounded from above and below by $C(1+\ro)$), our model differs from the models where the Bresch--Desjardin entropy plays an important role (see \cite{BrDe}); such a model was studied for chemically reacting mixtures in the evolutionary case in \cite{MPZ3} or \cite{XiXi}.

The pressure consists of two parts: the cold pressure and the molecular pressure; the latter (according to the Dalton law) is the sum of partial pressures $p_k$
$$  \pi = \pi_c+\pi_m = \ro^\gamma + \sumk p_k =  \ro^\gamma + \ro\theta \sumk\dfrac{Y_k}{m_k} $$
with $\gamma>1$ being constant. Note that, similarly as above, for the sake of simplicity, all relevant physical constants are set to be equal to $1$. As we restrict to the case $m_k=1$, we have simply
\begin{equation}
  \pi = \ro^\gamma + \ro\theta.  \label{pi}
  \end{equation}
Similarly, the specific internal energy
\begin{equation} e = e_c+e_m = \dfrac{1}{\gamma-1} \ro^{\gamma-1} + \theta \sumk c_{vk}Y_k \label{e}
\end{equation}
and the specific entropy
\begin{equation}
s_k = c_{vk} \log \theta - \log ( {\ro Y_k}   ),\qquad s= \sumk Y_k s_k ,\label{sk}
\end{equation}
where $c_{vk}$ is the specific heat at constant volume which is related to the specific heat at constant pressure by relation $ c_{pk} = c_{vk} + 1.$
Further, we introduce the specific enthalpies
\begin{equation} h_k = c_{pk} \theta,
\label{hk}
 \end{equation}
and the Gibbs function for each constituent
\begin{equation}
 g _k = h_k - \theta s_k , \qquad\qquad g= \sumk Y_k g_k.\label{gk}
 \end{equation}

For the diffusion fluxes $\mathbf{F}_k$ we assume the Fick law
\begin{equation}
 \mathbf{F}_k   = - D(\ro,\theta,\vek{Y}) \nabla Y_k,\qquad\quad 0<\ro D_0\leq D\leq C \ro \label{Fk}
 \end{equation}
with $D$ a continuous function. This is one of the assumptions which are quite simplifying, however, for our approach necessary. Consequently, we have
$$\sumk \mathbf{F}_k = \mathbf{0}.$$
The energy flux $
\mathbf{Q}$ is a sum of two parts
\begin{equation}
\mathbf{Q}=  \mathbf{q} +\sumk h_k \mathbf{F}_k , \label{Q}
\end{equation}
the one described by the well-known Fourier law
  \begin{equation}
\mathbf{q} = - \kappa (\ro,\theta)\nabla \theta, \label{q}
\end{equation}
 and the other representing the exchange of heat due to molecular diffusion by the product of specific enthalpies  $h_k$ and corresponding diffusion fluxes $\mathbf{F}_k$ defined above.

%%%%%%%%%%%%%%%%%%% BOUNDARY CONDITIONS  %%%%%%%%%%%%%%%%%%

The system is accompanied with the following boundary conditions on $\partial\Omega$, $\en$ denotes the outer normal vector to $\partial\Omega$,
\begin{align}
\mathbf{F}_k \cdot \en = \:& 0,\qquad k=1,\ldots,n \label{BC1}\\
 \mathbf{Q} \cdot \en  = \:&   L(\ro,\theta)(\theta-\Theta),\label{BC2}\\
 \u\cdot\en = \:& 0,\label{BC3}\\
 \en \cdot \field{S}(\ro,\nabla \u) \cdot \boldsymbol{\tau}^k + f \u \cdot \boldsymbol{\tau}^k = \:& 0,\label{BC4}
\end{align}
with $\boldsymbol{\tau}^k$, $k=1,2$, being two linearly independent vectors tangent to $\partial\Omega$, and $\Theta>0$ representing the outer temperature. We assume
\begin{equation}
f\geq 0,\quad \kappa,L \sim M(1+\theta^\alpha),\label{L}
\end{equation}
id est, we assume that the functions  $\frac{\kappa}{M (1+\theta^\alpha)}$ and $\frac{L}{M(1+\theta^\alpha)}$ are bounded away from zero and from above and the functions $\kappa$ and $L$ at least continuous.  Later on, we set $\alpha=3$. Note that the coefficient $M$ is the mean density of the fluid (closely connected to the total mass of the mixture) and will be assumed throughout this paper sufficiently large.  Hence, equivalently to \eqref{L}, we may assume that all quantities behave linearly with respect to the density. Our assumption $f \geq 0$ is connected with the assumption in the main theorem that our domain $\Omega$ must not be axially symmetric. In this case for all functions from $W^{1,p}(\Omega;\R^3)$ satisfying $\u \cdot \mathbf{n}=0$ on the boundary of $\Omega$ the Korn inequality yields
\begin{equation} \label{Korn1}
\|\u\|_{1,p} \leq C_p \|\field{D}(\u)\|_p
\end{equation}
for arbitrary $1<p<\infty$. We may replace the condition on geometric properties of $\Omega$ by assuming that $f>0$ and $f\sim M$. The proof follows similar lines, only in order to estimate the velocity, we need to combine the estimates from the entropy and the total energy balance and the proof becomes slightly more complicated, see also \cite{AMP16} and \cite{AMP18}.

 \begin{remark}
Integrating equation \refx{Yks}, by virtue of the boundary conditions \refx{BC1} and \refx{BC3}, one can observe that any stationary solution to our system has to satisfy sort of compatibility conditions
\begin{equation}
\inte{ \omega_k\bigl(\vec{Y}(x),\theta(x)\bigr)}=0,\quad \forall k=1,\ldots ,n . \label{compat}
\end{equation}
Then a natural question arises, whether this can be achieved by suitable choice of $\vek{Y}$, $\theta$ for any admissible choice of $\omega_k$. Let us show that it is indeed the case. In view of \refx{gk} condition \refx{omega} is reduced to
	\begin{equation}
	\sumk \omega_k(\vec{Y},\theta) \bigl(c_{pk}\theta  - c_{vk}\log\theta + \theta\log (Y_k)\bigr)\leq 0. \label{omeglog}
	\end{equation}
	We are able to construct two subsets $X_m$ and $Z_m$ of the domain of definition of $\omega_k$ such that for $m$ sufficiently large $\omega_k\geq 0$ on $X_m$ and  $\omega_k\leq 0$ on $Z_m.$ Let us define for fixed $\delta\in(0,\frac12)$ and $k\in\{1,\ldots,n\}$
	$$ X_m  := \Bigl\{ (\theta,\vek{Y}) \in  \mathbf{R}^{n+1}_+,\:\theta\in(\delta,\delta^{-1}), \: Y_k\leq \frac{1}{m},\: Y_i \in(\delta,1-\delta)\text{ for }i\neq k\Bigr\}.$$
	As $ \theta\log Y_k \leq  - \delta \log m $ on $X_m$ we have $g_k\to - \infty$ for $m\to + \infty$, as a matter of fact the other summands in \refx{omeglog} are bounded. Therefore, $\omega_k$ has to be non-negative on set $X_m$. Similarly,  on a set defined by
	$$ Z_m = \Bigl \{ (\theta,\vek{Y}) \in  \mathbf{R}^{n+1}_+,\: \theta\in(\delta,\delta^{-1}),\: Y_k = 1-\dfrac{n-1}{m}, Y_i =\dfrac{1}{m} \text{ for } i\neq k  \Bigr\} , $$
	we have
	\begin{align*}
	\sum_{l=1}^n \omega_l g_l = \omega_k g_k +  \sum_{i\neq k} \omega_i \Bigl(c_{pi}\theta  - c_{vi}\log\theta - \theta \log m \Bigr)\leq\: & 0\\
 \omega_k g_k +  \sum_{i\neq k} \omega_i \Bigl(c_{pi}\theta  - c_{vi}\log\theta \Bigr)\leq\: & \theta \log m \sum_{i\neq k} \omega_i .
	\end{align*}
Since the left-hand side of the last inequality is bounded independently of $m$, necessarily $\sum_{i\neq k} \omega_i \geq 0$ and in view of \refx{sumomeg} we have ultimately  $\omega_{k}\leq0$. This shows that $\omega_k$ changes the sign within its domain of definition (or is zero on a nontrivial subset), and thus there always exist suitable $\vek{Y}$ and $\theta$ such that condition \refx{compat} is satisfied.
\end{remark}

The main point of the preceding remark lies in the fact that by our construction, we do not need to consider condition \eqref{compat} in what follows. This condition will appear after the limit passages $\epsilon$ and $\delta \to 0$. Note that the situation here is the same as in \cite {GioPoZa}, \cite{PiaPo} or \cite{PiaPo2}, where, however, this problem was not discussed.

In what follows we use standard notation for the Sobolev and Lebesgue spaces with the abbreviate notation for their norms over $\Omega$, namely
$$  \norm{g}_{p}  := \norm{g}_{L^p(\Omega)} ,\qquad  \norm{g}_{k,p} :=\norm{g}_{W^{k,p}(\Omega)}.$$

Our main assumption, which is in fact crucial in the next considerations, concerns the density. We restrict ourselves to the case of fluids with high density in comparison to other given data. More precisely, we will write
$$\ro = M+r, \qquad \inte{r}=0,\qquad \frac{1}{\abs{\Omega}} \inte{\ro} = M,$$
we define for $p>3$ the following norm of the solution
$$ \Xi =  M^{\gamma-2}\norm{r}_{1,p} + \norm{\u}_{2,p} + \|\theta\|_{1,p} +\|\vek{Y}\|_{1,p}$$
and study the solutions with
\begin{equation}\Xi\ll M. \label{Assume}
\end{equation}
In particular, due to embedding $W^{1,p}(\Omega)\hookrightarrow L^\infty(\Omega)$, we will have $\frac{M}{2}<\ro<M.$ The existence of such solutions will be guaranteed provided
\begin{equation}
M> C_0(\ef,\Theta,\vek{\omega}), \label{choice}
\end{equation}
where the number $C_0$ on the right-hand side depends only on the given data, see \refx{MainIneq}.

Different models of mixtures have been recently studied in many papers. Our approach is close to the model introduced by Giovangigli (see \cite{Gio}), even though we simplify the quite general model, not only by considering the Fick law. The book also contains global-in-time existence proof for his model, however, only in case of small data. Some other (more sophisticated) small data results can be found in \cite{PiaShiZat1} and \cite{PiaShiZat2}.

Global-in-time solutions for large data are much more difficult to obtain. The first result in this direction is by Feireisl, Petzeltov\'{a} and Trivisa \cite{FePeTr}, where also the Fick law was studied. Further results, for more general diffusion matrix, can be found in \cite{Zat2}, \cite{Zat5} and in \cite{MPZ}, \cite{MPZ2}, \cite{MPZ3} or in \cite{XiXi}. In the papers, where also the fluid equations are included, the viscosity is similar to the viscosity studied here. However, here the estimates coming from Bresch--Desjardin entropy play an important role and therefore the precise form of viscosity which is unfortunately not very relevant from the point of view of physics must be used. The incompressible limit of these equations were studied in \cite{FePe} or in \cite{KwTr}.

In this paper we deal with steady solutions. The first studies of models based on the approach from \cite{Gio} are \cite{Zat5}. These results were extended in \cite{GioPoZa}, \cite{PiaPo}, \cite{PiaPo2} and in \cite{GuoX}. Note also that our approach is also based on the method of decomposition, developed and successfully used in nineties to study strong solutions of the compressible Navier--Stokes equations by Novotn\'{y} and Padula, see \cite{NoPa}. This fact will be hidden here as we can directly use results from \cite{AMP16} and \cite{AMP18} which are based on this technique and we do not have to prove them here again.

Last but not least, note that there are many different approaches to models of mixtures, see e.g. \cite{Bothe}, \cite{BuHa} or \cite{RaTao}. The models, where the velocities are studied separately for each constituent are studied in many papers by Mamontov and Prokudin, see e.g. \cite{MaPr1} or \cite{MaPr2}.

In what follows, we first show a priori estimates for our problem. They indicate the form of the main result which will be presented at the end of the following section. The last section will be devoted to the construction of a solution to our problem, i.e. to the proof of our main result.

\vspace*{0.7cm}

 %%%%%%%%%%%%%%%%%%%%%%%%%%%%%%%%%%%%%%%%%%%%%%%%%%%%
%%%%%%%%%%%%%%%%%%%%%%%%%%%%%%%%%%%%%%%%%%%%%%%%%%%%%%
%%%%%%%%%%      A PRIORI  ESTIMATES         %%%%%%%%%%
%%%%%%%%%%%%%%%%%%%%%%%%%%%%%%%%%%%%%%%%%%%%%%%%%%%%%%
 %%%%%%%%%%%%%%%%%%%%%%%%%%%%%%%%%%%%%%%%%%%%%%%%%%%%

\section{A priori estimates}

%%%%%%%%%%%%%%%%%%%%%%%%%%%%%%%%%%%%%%%%%%%%%%%%%%%%%%
%%%%%%%%%%%%%%%%%%%%%%%%%%%%%%%%%%%%%%%%%%%%%%%%%%%%%%

This section contains the a priori estimates for our system. Note that the following lemma only indicates in which spaces we should look for solutions; it is even not a rigorous proof of a priori estimates since we use during the computations the assumption that the "density perturbation" $r \ll M$ which does not follow from the assumptions and must be verified in the following section, where the solution is constructed.

\begin{lem}
  Suppose that the quadruple $(\u,r,\theta,\vek{Y})$ is a smooth solution to system \refx{CEs}--\refx{Yks} satisfying boundary conditions \refx{BC1}--\refx{BC4}. Assume further that the solution lies in the class \refx{Assume}. Then we have
  \begin{align}
     \norm{\u}_{1,2}  + \norm{\theta}_{9}+\norm{\theta}_{1,2} + \|\vek{Y}\|_{1,2}  \leq &\: E(\ef,\Theta), \label{Apriori1}\\
   \norm{\u}_{2,p} +  M^{\gamma-2}\norm{r}_{1,p} + \norm{\theta}_{1,p} +\|\vek{Y}\|_{1,p} \leq &\: C(\ef,\Theta,\omega_k),\label{Apriori2}
   \end{align}
   with the right-hand sides independent of the solution and $M$.
\end{lem}

\begin{proof}
First, we multiply equation \refx{EEs} by $\dfrac{1}{\theta}$; it yields
\begin{equation*}
{\dfrac{\field S :\field D(\u)}{\theta}}- {\dfrac{\Div \mathbf{Q}}{\theta} }=  {\dfrac{\pi(\ro,\theta)\Div\u}{\theta}}+{\dfrac{\Div(\ro e \u)}{\theta}},
\end{equation*}
hence by constitutive assumptions \refx{S}, \refx{pi}, \refx{e} and \refx{Q} we have
\begin{equation}
\dfrac{2\ro\abs{\field D(\u)}^2}{\theta} - \dfrac{ \mathbf{q}\cdot \nabla \theta }{\theta^2} - {\Div\Bigl(\dfrac{ \mathbf{Q}}{\theta}\Bigr)}= {\dfrac{ \sumk h_k \mathbf{F}_k \cdot \nabla\theta }{\theta^2}} +{\dfrac{ \ro^\gamma \Div\u}{\theta} } +{ \ro \Div\u }+ {\dfrac{\ro  \u \cdot \nabla (e_c+e_m)}{\theta}}. \label{entr}
\end{equation}
Further, note that
\begin{multline*}
{\dfrac{ \pi_c\Div\u}{\theta} } +{\dfrac{\ro  \u \cdot \nabla e_c}{\theta}} = {\dfrac{ 1}{\theta}\biggl(\ro^\gamma \Div\u  + \ro  \u \cdot \nabla\bigl( \dfrac{\ro^{\gamma-1}}{\gamma-1}\bigr)\biggr)}  ={\dfrac{ 1}{\theta}\biggl(\ro^{\gamma-1} \ro\Div\u  + \ro^{\gamma-1} \u  \cdot \nabla\ro \biggr)} \\ ={\dfrac{ \ro^{\gamma-1}}{\theta}\Div( \ro\u)} =0
\end{multline*}
and
\begin{equation}
{\dfrac{ \pi_m\Div\u}{\theta} } + {\dfrac{\ro  \u \cdot \nabla e_m}{\theta}} = {\ro \Div\u  +\ro \sumk c_{vk}Y_k  \u \cdot \dfrac{\nabla \theta}{\theta}+\ro  \u \cdot  \sumk c_{vk}\nabla Y_k }.\end{equation}
Next, by multiplying equation \refx{Yks} by $\dfrac{g_k}{\theta}$, we obtain due to \refx{hk}--\refx{gk}
\begin{equation}
{\Div(\ro Y_k \u)\dfrac{h_k-\theta s_k}{\theta} }+ {\Div \mathbf{F}_k \dfrac{g_k}{\theta} } = {\dfrac{\ro \omega_k g_k}{\theta}}
\end{equation}
which reads
\begin{equation}
{\ro  \u \cdot c_{pk} \nabla Y_k  } - {\Div(\ro Y_k \u) s_k} + {\Div \Bigl(\mathbf{F}_k \dfrac{g_k}{\theta}\Bigr) }  -{\mathbf{F}_k\cdot\nabla \Bigl( \dfrac{g_k}{\theta}\Bigr) }   = {\ro \omega_k \dfrac{g_k}{\theta}}.\label{Fkgk}
\end{equation}
Moreover, as $c_{pk} = c_{vk}+1$, we have
\begin{equation*}
\sumk{\ro  \u \cdot c_{pk} \nabla Y_k  } = \sumk {\ro  \u \cdot c_{vk} \nabla Y_k  } +  {\ro  \u \cdot  \nabla 1  }=
\sumk {\ro  \u \cdot c_{vk} \nabla Y_k  },
\end{equation*}
and
$$ \sumk{\Div(\ro  \u  Y_k)s_k} = {\Div(\ro s  \u ) }  - {\ro  \u  \sumk Y_k \cdot\nabla s_k} . $$
The last term can be expanded by relation \refx{sk} as follows
\begin{equation*}
\begin{split}
{\ro  \u  \sumk Y_k \cdot\nabla s_k} &=
{\ro  \u  \sumk Y_k \cdot\nabla ( c_{vk}\log\theta - \log \ro - \log{Y_k})} \\
&=   {\ro  \u  \sumk Y_k \cdot\Bigl( c_{vk}\dfrac{\nabla\theta}{\theta} - \dfrac{\nabla\ro}{\ro} - \dfrac{\nabla Y_k}{Y_k}\Bigr)}\\
&=   {     \dfrac{ \sumk  c_{vk} Y_k \ro  \u\cdot\nabla\theta}{\theta} -   \u  \cdot  \nabla\ro },
\end{split}
\end{equation*}
where we have used that $\sumk Y_k=1$. Similarly, as $\sumk\mathbf{F}_k=\mathbf{0}$, we have
\begin{multline*}
-\sumk{\mathbf{F}_k\cdot\nabla \Bigl( \dfrac{g_k}{\theta}\Bigr) } = -\sumk {\mathbf{F}_k\cdot\nabla \Bigl( \dfrac{h_k-\theta s_k}{\theta}\Bigr) } = \sumk {\mathbf{F}_k\cdot\nabla  s_k } =  \sumk {\mathbf{F}_k\cdot \Bigl(  c_{vk} \dfrac{\nabla \theta}{\theta} - \dfrac{\nabla \ro}{\ro} - \dfrac{\nabla Y_k}{Y_k}  \Bigr)}\\
 =   {\sumk \dfrac{\mathbf{F}_k\cdot   c_{vk} \nabla \theta}{\theta}  } - {\sumk \dfrac{\mathbf{F}_k  \cdot \nabla Y_k}{Y_k}  } .
\end{multline*}
Thus, integrating \refx{Fkgk} over $\Omega$ yields
\begin{multline*}
 \inte{\Div\Bigl(\sumk \mathbf{F}_k \dfrac{g_k}{\theta} - \ro s  \u  \Bigr) } + \inte{\sumk \dfrac{\mathbf{F}_k\cdot   c_{vk} \nabla \theta}{\theta}  } - \inte{\sumk \dfrac{\mathbf{F}_k  \cdot \nabla Y_k}{Y_k}  }  \\
 =  \inte {\u  \cdot  \nabla\ro } - \inte{\sumk \ro  \u \cdot c_{vk} \nabla Y_k  }- \inte{     \dfrac{ \sumk  c_{vk} Y_k \ro  \u\cdot\nabla\theta}{\theta}}   +\inte{\sumk\ro \omega_k \dfrac{g_k}{\theta}},
\end{multline*}
while the integrated version of \refx{entr} reads
\begin{multline*}
\inte{\biggl(\dfrac{2\ro\abs{\field D(\u)}^2}{\theta} - \dfrac{ \mathbf{q}\cdot \nabla \theta  }{\theta^2}\biggr)} - \inte{\Div\Bigl(\dfrac{ \mathbf{Q}}{\theta}\Bigr)}\\= \inte{\dfrac{ \sumk h_k \mathbf{F}_k \cdot \nabla\theta }{\theta^2}}  +\inte{ \ro \Div\u }+ \inte{\dfrac{\ro  \u \cdot \nabla \bigl(\theta \sumk c_{vk}Y_k\bigr)}{\theta}}.
\end{multline*}
We conclude (using the Fourier law \refx{q})
\begin{multline*}
\inte{\biggl(\dfrac{2\ro\abs{\field D(\u)}^2}{\theta} + \dfrac{\kappa(\ro,\theta)\abs{\nabla \theta}^2}{\theta^2}\biggr)}  - \inte{\sumk \dfrac{\mathbf{F}_{k}  \cdot \nabla Y_k}{Y_k}  } +\inte{\Div\Bigl(-\sumk \mathbf{F}_{k}{s_k} -   \ro s  \u + \dfrac{\kappa \nabla \theta}{\theta}  \Bigr)}\\  = \inte{\sumk\ro \omega_k \dfrac{g_k}{\theta}}.
\end{multline*}
%%%%%%%%%%%%%%%%%%%%%%%%%%   ENTROPY BALANCE %%%%%%%%%%%%%%%%%%%%%%%%%%%%%%%%%
Alternatively, we in fact got the entropy balance
\begin{equation*}
\inte{\Div\Bigl( \ro s  \u +\dfrac{ \mathbf{q}}{\theta} +  \sumk s_k \mathbf{F}_k \Bigr)}  = \inte{ \sigma}
\end{equation*}
with $\sigma$ the entropy production rate
$$\sigma = \dfrac{2\ro\abs{\field D(\u)}^2}{\theta} + \dfrac{\kappa (\ro,\theta)\abs{\nabla \theta}^2}{\theta^2}-\sumk \dfrac{\mathbf{F}_k  \cdot \nabla Y_k}{Y_k} -\sumk\ro \omega_k \dfrac{g_k}{\theta} . $$
Thus, by virtue of integrating by parts on the left-hand side we have, due to boundary conditions \refx{BC1}--\refx{BC3},
\begin{multline*}
 \inte{ \biggl(\dfrac{2\ro\abs{\field D(\u)}^2}{\theta} + \dfrac{\kappa (\ro,\theta)\abs{\nabla \theta}^2}{\theta^2}-\sumk \dfrac{\mathbf{F}_k  \cdot \nabla Y_k}{Y_k} -\sumk\ro \omega_k \dfrac{g_k}{\theta}\biggl)}  + \inth {L(\ro,\theta) \dfrac{\Theta}{\theta} } \\ = \inth {L(\ro,\theta) },
\end{multline*}
and due to \refx{omega}, \refx{Fk} and \refx{L}
\begin{equation}
 \inte{ \biggl( \dfrac{\abs{\field D(\u)}^2}{\theta} + \abs{\nabla \log\theta}^2+\theta^{\alpha-2}\abs{\nabla \theta}^2+\sumk  \dfrac{  \abs{\nabla Y_k}^2}{Y_k} \biggr) } + \inth {\Bigl(\dfrac{1}{\theta}+\theta^{\alpha-1} \Bigr)}  \leq C \inth {\theta^\alpha }. \label{entrop}
\end{equation}
Moreover, $  \theta^{\alpha-2}\abs{\nabla \theta}^2 = \abs{ \theta^{\frac{\alpha}2-1}\nabla \theta }^2 = \abs{ \frac{2}{\alpha} \nabla( \theta^{\frac{\alpha}{2}})}^2$, hence
$$ \norm{\theta}_{3\alpha}^\alpha = \norm{\theta^{\alpha/2}}_{6}^2 \leq C \norm{\theta^{\alpha/2}}_{1,2}^2 \leq C\biggl(\norm{\nabla(\theta^{\alpha/2})}_{2}^2  +  \inth{\theta^\alpha} \biggr)  \leq C \norm{\theta}_{L^{\alpha}(\partial\Omega)}^\alpha, $$
and (see \eqref{Korn1})
$$ \norm{\u}_{1,\frac{6\alpha}{3\alpha+1}} \leq C\|\field{D}(\u)\|_{\frac{6\alpha}{3\alpha+1}}\leq C \norm{\dfrac{\field{D}(\u)}{\sqrt{\theta}}}_{2} \norm{\sqrt{\theta}}_{6\alpha}
\leq \norm{\dfrac{\field{D}(\u)}{\sqrt{\theta}}}_{2} \norm{\theta}_{3\alpha}^{1/2} \leq  C \norm{\theta}_{L^{\alpha}(\partial\Omega)}^{\frac{\alpha}{2}+\frac12} .$$
Furthermore, we have the total energy balance
$$ \inth{ \bigl( L(\ro,\theta)(\theta-\Theta) + f \abs{\u}^2 \bigr) }  = \inte{ \ro\mathbf{f}\cdot\u}  ,$$
hence
\begin{equation*}
 \inth{  \theta^{\alpha+1} }   \leq C \Bigl(\inte{ \mathbf{f}\cdot\u}  + \inth{\theta^{\alpha}\Theta} \Bigr) ,
 \end{equation*}
  which yields
\begin{equation}
\inth{\theta^{\alpha+1} }   \leq C \big( \norm{\mathbf{f}}_{\frac{6\alpha}{5\alpha-1}} \norm{\u}_{\frac{6\alpha}{3\alpha+1}} + \norm{\theta}_{L^{\alpha+1} (\partial \Omega)}^\alpha \|\Theta\|_{L^{\alpha+1}(\partial\Omega)} \big) . \label{total}
\end{equation}
Therefore, combining inequalities \refx{entrop} and \refx{total} we conclude
\begin{equation} \label{veloc}
\norm{\u}_{\frac{6\alpha}{\alpha+1}}^2 \leq C \norm{\u}_{1,\frac{6\alpha}{3\alpha+1}}^2 \leq C \big( \norm{\mathbf{f}}_{\frac{6\alpha}{5\alpha-1}}^2  +  \|\Theta\|_{L^{\alpha+1}(\partial\Omega)}^{\alpha+1} \big).
\end{equation}
Further, from \eqref{total} we also have
$$
\inth{\theta^{\alpha+1} } \leq C \big( \norm{\mathbf{f}}_{\frac{6\alpha}{5\alpha-1}}^2  +  \|\Theta\|_{L^{\alpha+1}(\partial\Omega)}^{\alpha+1} \big)
$$
which yields
\begin{equation} \label{tempera}
\norm{\theta}_{3\alpha}^\alpha + \|\nabla (\theta^{\frac \alpha 2})\|_2^2 \leq C\big( \norm{\mathbf{f}}_{\frac{6\alpha}{5\alpha-1}}^2  +  \|\Theta\|_{L^{\alpha+1}(\partial\Omega)}^{\alpha+1} \big)^\frac{\alpha}{\alpha+1}.
\end{equation}
In particular, taking $\alpha=3$
\begin{equation} \label{velo2}
\norm{\u}_{9/2}^2 \leq C  \big(\norm{ \mathbf{f}}_{9/7}^2  + \norm{\Theta}_{L^{4}(\partial\Omega)}^{4} \big)
\end{equation}
as well as
\begin{equation} \label{tempera2}
\|\theta\|_9^3 + \|\nabla \theta^{\frac 32}\|_2^2 \leq C \big(\norm{ \mathbf{f}}_{9/7}^{\frac 32}  + \norm{\Theta}_{L^{4}(\partial\Omega)}^{3} \big).
\end{equation}

Besides, we can go back to the momentum equation, and test it by the velocity field to get
$$ M  \norm{\nabla\u}_2^2+M\|\u\|_{L^2(\partial \Omega)}^2 \leq C \inte{ \bigl(\nabla(\ro\theta)\cdot\u + \ro\ef\cdot\u \bigr)}\leq C \norm{\ro}_\infty \bigl( \norm{\theta}_2\norm{\Div\u}_2+\norm{\ef}_{6/5}\norm{\u}_6\bigr) , $$
\begin{equation}
  \norm{\u}_{1,2}^2 \leq C (\Theta,\ef).
\end{equation}
Therefore, there exists a constant $E$ dependent only on given data, such that (recall that $0\leq Y_k\leq 1$)
\begin{equation}
  \norm{\u}_{1,2}  + \norm{\theta}_{9}+\norm{\theta}_{1,2} + \|\vek{Y}\|_{1,2}   \leq E
\end{equation}
and the first part of the estimates is proved.

   %%%%%%%%%%%%%%%%%%%      %%%%%%%%%%%%%%%%%%%  %%%%%%%%%%%%%%%%%%%   %%%%%%%%%%%%%%%%%%%
%%%%%%%%%%%%%%%%%%%%%%%%%%%% METHOD OF DECOMPOSITION %%%%%%%%%%%%%%% %%%%%%%%%%%%%%
 %%%%%%%%%%%%%%%%%%%        %%%%%%%%%%%%%%%%%%%     %%%%%%%%%%%%%%%%%%%   %%%%%%%%%   %%%   %%% %%   %   %

In order to improve the regularity we use the method of decomposition in the same way as for the case of single-component fluid in \cite{AMP16} and \cite{AMP18}. Here, we use the fact that the momentum equation is influenced by the chemical reactions only through the temperature involved in the pressure. Otherwise, all the other terms in the momentum equation are identical to the single-component case, and thus can be treated exactly in the same way. This yields (cf. \cite[Lemma 2.1, p. 268]{AMP18}),
\begin{equation} \label{appest1}
 \norm{\u}_{2,p} + M^{\gamma-2}\norm{r}_{1,p} \leq C\bigl(\norm{\ef}_p + \norm{\nabla \theta}_p\bigr).
 \end{equation}
Let us now estimate the first gradient of the temperature. The equation for the internal energy can be rewritten in the form
$$-\Div\bigl( \kappa(\ro,\theta) \nabla \theta\bigr) = \field{S}:\nabla\u - \pi_m \Div\u - \ro \u \cdot\nabla e_m - \Div\Bigl(\sumk\theta c_{pk} \mathbf{F}_k  \Bigr)  $$
and by regularity of this elliptic equation, see \cite{DM95}, we get
\begin{equation*}
\begin{split}
M\norm{\theta}_{1,p} &\:\leq C \Bigl( M\norm{ \theta- \Theta}_{W^{-1/p,p}(\partial\Omega)} +  \norm{\field{S}:\nabla\u}_{\frac{3p}{3+p}} + \norm{ \pi_m \Div\u+\ro\u \cdot\nabla e_m}_{\frac{3p}{3+p}} \Bigr. \\
&\qquad\qquad\qquad\qquad\qquad\qquad\qquad\qquad\qquad\qquad\qquad\qquad\qquad \Bigl. +  \Bigl\|\sumk\theta c_{pk} \mathbf{F}_k\Bigr\|_p  \Bigr) \\
\intertext{which yields}
\norm{\theta}_{1,p} &\:\leq C \biggl( \norm{ \theta- \Theta}_{W^{-1/p,p}(\partial\Omega)} +  \norm{\nabla\u}_{\frac{6p}{3+p}}^2 + \norm{ \theta \Div\u}_{\frac{3p}{3+p}} +\norm{ \u \cdot\nabla \theta}_{\frac{3p}{3+p}} \\ &\qquad\qquad\qquad\qquad\qquad\qquad\qquad\qquad\qquad\qquad +\sumk \Bigl (\norm{ \u  \theta\cdot\nabla Y_k }_{\frac{3p}{3+p}}+ \frac1M\norm{ \theta   \mathbf{F}_k}_p   \Bigr) \biggr).
\end{split}
\end{equation*}
Note that to prove the estimates, we write $\Div\bigl( \kappa(\ro,\theta) \nabla \theta\bigr) = \Div\bigl( M(1+\theta^3) \nabla \theta\bigr) + \Div\bigl( r (1+\theta^3) \nabla \theta\bigr)$ and use \cite[Lemma A.2 and Lemma A.3]{AMP18} together with the assumption that $\|r\|_\infty \ll M$ (this assumption has already been used once, in order to obtain \eqref{appest1}).

Let us estimate the terms on the right-hand side of the preceding inequality. We need to carefully interpolate between estimates coming from the entropy balance and the higher order norms. Using the H{\"o}lder and the Gagliardo--Nirenberg interpolation inequalities we obtain
\begin{equation}
\norm{\nabla\u}_{\frac{6p}{3+p}}^2 \leq C \norm{\nabla\u}_{2}^{\frac{6p-6}{5p-6}} \norm{\u}_{2,p}^{\frac{4p-6}{5p-6}} \leq C \norm{\u}_{2,p}^{\frac{4p-6}{5p-6}},
 \end{equation}
\begin{equation}
\norm{\theta \Div \u}_{\frac{3p}{3+p}} \leq C \norm{\theta}_{9} \norm{ \nabla \u}_{\frac{9p}{9+2p}} \leq C \norm{\theta}_{9} \norm{ \nabla \u}_{2}^{1-t} \norm{  \u}_{2,p}^{t},
\end{equation}
\begin{equation}
\norm{\u \nabla \theta}_{\frac{3p}{3+p}} \leq C \norm{\u}_{6} \norm{ \nabla \theta}_{\frac{6p}{6+p}} \leq C \norm{\u}_{1,2}  \norm{ \nabla \theta }_{2}^{1-t'} \norm{ \nabla\theta}_{p}^{t'}
\end{equation}
with $t<1$ and $t'<1$. As also $\frac{4p-6}{5p-6}<1$, these three terms can be put directly to the left-hand side  by means of the Young inequality.
Further, we will use
$$\norm{\u}_\infty\leq C \norm{\u}_6^{1-r}\norm{\u}_{2,p}^r  \qquad\qquad  \norm{\theta}_\infty\leq C \norm{\theta}_9^{1-s}\norm{\theta}_{1,p}^s ,$$
where $r = \dfrac{p}{5p-6}$ and $s = \dfrac{p}{4p-9}.$ Thus,
\begin{equation}
\norm{ \u  \theta\cdot\nabla Y_k }_{\frac{3p}{3+p}} \leq \norm{ \u}_\infty\norm{\theta}_{9}\norm{\nabla Y_k }_{\frac{9p}{9+2p}} \leq C \norm{\u}_{2,p}^r  \norm{\nabla Y_k }_{\frac{9p}{9+2p}}.
\end{equation}
For $3<p\leq\frac{18}{5}$ we have $\frac{9p}{9+2p} \leq 2$ and therefore, we can put the term to the left-hand side. Otherwise, we interpolate
$$\norm{\nabla Y_k }_{\frac{9p}{9+2p}}  \leq C\norm{\nabla Y_k }_{2}^{1-\varsigma} \norm{\nabla Y_k }_{p}^\varsigma$$
with $\varsigma = \dfrac{5p-18}{9p-18}.$
Similarly,
$$\norm{ \theta   \mathbf{F}_k}_p  \leq \norm{\theta}_\infty \norm{   \mathbf{F}_k}_{p} \leq C   \norm{\theta}_{1,p}^s \norm{   \mathbf{F}_k}_p . $$
Summing up the previous estimates yields
\begin{equation}
\norm{\theta}_{1,p} \leq C(\ef,\Theta) \Bigl(1+ \norm{\nabla Y_k }_{p}^{\frac{5p-6}{4p-6}\frac{ 5p-18 } { 9p-18} }   + \frac1M \norm{   \mathbf{F}_k}_p^{   \frac{4p-9}{3p-9}} \Bigr).
\end{equation}
Further, note that
$$ M\norm{\nabla Y_k}_p  \leq C \norm{\mathbf{F}_k}_p$$ and
$$  \frac{5p-6}{4p-6}\frac{ 5p-18 } { 9p-18} < \frac{4p-9}{3p-9} \text{ for any }p>\frac{18}{5},$$ whence
\begin{equation}
\norm{\theta}_{1,p} \leq C(\ef,\Theta) \Bigl(1  + \frac1M \norm{   \mathbf{F}_k}_p^{   \frac{4p-9}{3p-9}} \Bigr). \label{thetaFk}
\end{equation}
Therefore, it remains to bound $\mathbf{F}_k$.  Equation \refx{Yks} reads
$$ \Div\mathbf{F}_k = \ro  \omega_k -\Div(\ro  \u \cdot Y_k),$$
and we can estimate (we use again that $\|r\|_\infty \ll M$; note also that we work with solutions such that $0\leq Y_k\leq 1$, $k=1,2,\dots,n$)
\begin{equation}
 \norm{\nabla Y_k}_{p}\leq C\bigl( \norm{ \omega_k}_{p} + \norm{\u}_{p} \bigr)\leq
C\bigl( 1+ \norm{\u}_{2,p}^q \bigr)
\leq
C\bigl( 1+ \norm{\theta}_{1,p}^q \bigr)
\end{equation}
for $p>6$, while
\begin{equation}
 \norm{\nabla Y_k}_{p}\leq C
\end{equation}
for $3<p\leq 6$,
where we used
$$ \norm{\u}_{p} \leq \norm{\u}_{6}^{1-q} \norm{\u}_{2,p}^q  \text{ with } q = \dfrac{p-6}{5p-6}$$
for $p>6$. Hence, by virtue of the Young inequality we obtain
\begin{equation}
 \frac{1}{M} \norm{\mathbf{F}_k}_p\leq C  \norm{\nabla Y_k}_{p} \leq C \bigl( 1+ \norm{\theta}_{1,p}^{\frac{p-6}{5p-6}} \bigr)
\end{equation}
if $p>6$, while
\begin{equation}
 \frac{1}{M} \norm{\mathbf{F}_k}_p\leq C
\end{equation}
if $3<p\leq 6$.
Together with inequality \refx{thetaFk} this yields
\begin{equation}
\|\theta\|_{1,p} \leq C
\end{equation}
if $3<p\leq 6$, while
\begin{equation}
\norm{\theta}_{1,p} \leq C(\ef,\Theta,\omega) \Bigl(1  +  \norm{  \theta}_{1,p}^{   \frac{4p-9}{3p-9}\frac{p-6}{5p-6}} \Bigr)
\end{equation}
for $p>6$.
However,
$$ \frac{4p-9}{3p-9}\frac{p-6}{5p-6}  <1\text{ for all }p>6.$$
Therefore, we can conclude
\begin{equation}
\norm{\u}_{2,p} + M^{\gamma-2}\norm{r}_{1,p} +\norm{\theta}_{1,p}+\|\vek{Y}\|_{1,p} \leq C(\ef,\Theta,\vek{\omega}). \label{MainIneq}
\end{equation}
\end{proof}

The last estimate of the solutions determines the class within which we search for the solution, thus we are ready to formulate our main result.

% % % %% % % %% % % %% % % %% % % %
% % % % % %  Theorem 1  % % % % % %
% % % %% % % %% % % %% % % %% % % %

\begin{thm}\label{main}
	Let $\gamma >1$.
	Let $\Omega\subset \field{R}^3$ be a $C^2$ bounded domain which is not axially symmetric. Let $p>3$,
	$\ef\in L^{p}(\Omega)$, $\Theta_0\in W^{1-1/p,p}(\partial\Omega)$, $\Theta_0\geq T_0>0$ a.e. on $\partial \Omega$, $\vec \omega \in L^{\infty}(\Omega;\R^n)$. Then there exists $M_0$ sufficiently
	large with respect to  $\norm{\ef}_p$, $\norm{\vek{\omega}}_{\infty}$ $\norm{\Theta}_{W^{1-1/p,p}(\partial\Omega)}$  in the sense of condition
	\refx{choice} such that for any $M\geq M_0$, where $M=\frac{1}{\Omega} \int_{\Omega} \ro\, {\rm d}x$
	is the average of the  density, there exists a strong stationary solution to the
 system \refx{CEs}--\refx{BC4}. Moreover, it possesses the regularity
	$(\ro,\u,\theta, \vek{Y})\in W^{1,p}(\Omega)\times W^{2,p}(\Omega;\field{R}^3)\times W^{2,p}(\Omega)\times W^{2,p}(\Omega;\field{R}^n).$
\end{thm}

Note that the preceding computations were rather instructive, they do not provide any proof of the result. Moreover, we used the assumption that $\|r\|_\infty \ll M$ which was not guaranteed above and must be shown independently while proving the theorem.

  %%%%%%%%%%%%%%%%%%%%%%%%%%%%%%%%%%%%%%%%%%%%%%%%%%%%%%%%%%%%%%%

  %%%%%%%%%%%%%  CONSTRUCTION OF THE SOLUTIONS  %%%%%%%%%%%%%%%%%

  %%%%%%%%%%%%%%%%%%%%%%%%%%%%%%%%%%%%%%%%%%%%%%%%%%%%%%%%%%%%%%%

\section{Construction of the solutions}

In contrast to the single-constituted fluid case studied in \cite{AMP16} or \cite{AMP18}, we are not able to deduce here the comparison principle for the temperature. Therefore, inspired by previous works on compressible mixtures and in order to guarantee the non-negativity of temperature as well as the specific concentrations, we introduce two approximation parameters $\epsilon,\delta>0$. At the first stage we search rather for the logarithms than for the quantities themselves.

For the construction of the solutions we will use the homotopy argument due to Leray and Schauder from their celebrated paper \cite{LeSc34}. It states that a family of mappings $\mathcal{F}_\lambda$ on a Banach space $X$ has at least one fixed point for each value of $\lambda\in[0,1]$, as soon as the following three conditions are satisfied:
\begin{enumerate}
\item For $\lambda=0$ the existence and uniqueness of the fixed point is guaranteed.
\item Mapping $\mathcal{F}_\lambda$: $X\to X$ is continuous and compact for each $\lambda \in [0,1]$, and depends continuously on $\lambda$.
\item We have uniform a priori bounds on the fixed points of $\mathcal{F_\lambda}$, id est, on solutions to $\mathcal{F_\lambda} (x) =x$ for any $\lambda \in [0,1]$.
\end{enumerate}
We set in our situation
$$
X = \{(\u, \theta, \vec{Y}), \u \in W^{1,\infty}(\Omega;\R^3); \theta, \log \theta \in W^{1,\infty}(\Omega); \vec Y, \{\log Y_k\}_{k=1}^n \in W^{1,\infty}(\Omega;\R^n)\}.
$$

For fixed $\epsilon$  and $\delta$ positive we will aim to apply the theorem on the family
$$\mathcal{F}_\lambda\colon (\overline{\u},\overline{\theta}, \overline{\vec Y}) \mapsto ({\u},{\theta}, {\vec Y})$$
 with $\mathcal{F}_\lambda$ being a solution operator of the following system
\begin{align}
     \Div(\ro \u) = &\:  0,\label{eq47}  \\
     \lambda \ro \overline{\u} \cdot\nabla   \u - \Div \field S(\ro,\nabla\u)  +   \nabla  \pi\big(\ro,\lambda\tfrac{\overline{\theta}}{g(\|\overline\theta\|_{1,p})}\big)=&\: \ro\mathbf{f},  \label{eq48} \\
    -\Div\Bigl( \frac{\delta + \theta}{\theta} \overline{\kappa}_\lambda \nabla \theta\Bigr) =\field{S}(\ro,\nabla \u):\nabla\u &\: \nonumber \\
		-  \lambda \ro  \frac{\overline{\theta}}{g(\|\overline\theta\|_{1,p})}\Div\u - &\:  \lambda \ro \u \cdot\nabla \widetilde{e}_m(\ro,\overline{\theta},\overline{\vec Y})- \lambda \Div\bigl(\sumk\overline{\theta} c_{pk} {\mathbf{F}}_{k\lambda}  \bigr)   , \label{eq49}\\
 \lambda \frac{\ro \overline{\u} \cdot\nabla \overline{Y_k}}{g(\|\overline\theta\|_{1,p})}    + \Div \bigl(\mathbf{F}_{k\lambda}  \bigr)  =\: \lambda \ro \omega_k\bigl(\overline{\vec Y},\overline{\theta}\bigr) - \epsilon \log Y_k - &
  \delta  Y_k  +  \dfrac{\delta}{n}, \quad k=1,2,\ldots,n, \label{eq50}
\end{align}
where
$$
g(\|\overline\theta\|_{1,p}) = \max\big\{1,\tfrac{\|\overline{\theta}\|_{1,p}}{C_0}\big\}
$$
with $C_0$ defined in \eqref{MainIneq}, 
\begin{equation} \label{eq50.a}
\widetilde{e}_m(\ro,\overline{\theta},\overline{\vec{Y}}) = \frac{\overline{\theta}\sum_{k=1}^n c_{vk} \overline{Y}_k}{g(\|\overline\theta\|_{1,p})},
\end{equation}
and the corresponding boundary conditions
\begin{align}
\mathbf{F}_{k\lambda} \cdot \en = \:& 0, \quad k=1,2,\ldots,n, \\
- \frac{\delta + \theta}{\theta} \overline{\kappa}_\lambda \nabla \theta \cdot \en  = \:&   L_\lambda(\ro,\overline{\theta})(\theta-\Theta) + \epsilon \log \theta ,\\
 \u\cdot\en = \:& 0, \label{uen}\\
 \en \cdot \field{S}(\ro,\nabla \u) \cdot \boldsymbol{\tau}^l + f \u \cdot \boldsymbol{\tau}^l = \:& 0, \quad l=1,2, \label{slip}
\end{align}
where we have denoted  $\overline{\kappa}_\lambda=\kappa_0M(1+\theta^3)+\lambda (\kappa(\ro,\overline{\theta}) - \kappa_0 M(1+\theta^3))$, $D_\lambda = D_0M+\lambda (D - D_0M),$ and similarly for $L_\lambda = L_0M(1+\theta^3)+\lambda (L - L_0M(1+\theta^3))$ ($L_0$ and $D_0$ are positive numbers), $${\mathbf{F}}_{k\lambda} = - \Bigl(D_\lambda(\ro,\overline{\theta},\overline{\vek{Y}}) + \dfrac{\epsilon}{Y_k}\Bigr) \nabla Y_k .$$
As mentioned above, we need to reformulate the problem in terms of logarithms of temperature and concentrations, this reads
\begin{align}
     \Div(\ro \u) = &\:  0,\label{eq7}  \\
     \lambda \ro \overline{\u} \cdot\nabla   \u - \Div \field S(\ro,\nabla \u)  +   \nabla  \pi\big(\ro,\lambda\frac{\mathrm{e}^{\overline{z}}}{g(\|\mathrm{e}^{\overline{z}}\|_{1,p})}\big)=&\: \ro\mathbf{f},  \label{eq8} \\
    -\Div\Bigl( (\delta + \mathrm{e}^{z}) \overline{\kappa}_\lambda \nabla z\Bigr) =\field{S}(\ro,\nabla\u) :\nabla\u  \nonumber\\
		-\lambda \ro  \frac{\mathrm{e}^{\overline{z}}}{g(\|\mathrm{e}^{\overline{z}}\|_{1,p})}\Div\u - &\: \lambda \ro \u \cdot\nabla \widetilde{e}_m\bigl(\ro,\mathrm{e}^{\overline{z}},\mathrm{e}^{\overline{\vec w}}\bigr) - \lambda \Div\Bigl(\sumk\mathrm{e}^{\overline{z}} c_{pk} {\mathbf{F}}_{k\lambda}  \Bigr)   , \label{eq9}\\
 \epsilon w_k- \Div \biggl( \Bigl(D_\lambda\mathrm{e}^{w_k} +\epsilon \Bigr) \nabla w_k \biggr)  =&\:-\lambda \frac{\ro \overline{\u} \cdot\nabla  \mathrm{e}^{\overline{w_k}}}{g(\|\mathrm{e}^{\overline{z}}\|_{1,p})} + \lambda \ro \omega_k\bigl(\mathrm{e}^{\overline{\vec w}},\mathrm{e}^{\overline{z}}\bigr)  -  \delta  \mathrm{e}^{w_k}  +  \dfrac{\delta}{n},  \label{eq0}
\end{align}
$k=1,2,\ldots,n$,
with corresponding boundary conditions
\begin{align}
-\Bigl(D_\lambda\mathrm{e}^{w_k} +\epsilon \Bigr) \nabla w_k  \cdot \en = \:& 0 , \quad k=1,2,\ldots,n,\\
- (\delta + \mathrm{e}^z) \overline{\kappa}_\lambda \nabla z \cdot \en  = \:&   L_\lambda(\ro,\mathrm{e}^{z} )(\mathrm{e}^z-\Theta) + \epsilon z,\\
 \u\cdot\en = \:& 0, \\
 \en \cdot \field{S}(\ro,\nabla \u) \cdot \boldsymbol{\tau}^l + f \u \cdot \boldsymbol{\tau}^l = \:& 0, \quad l=1,2.
\end{align}

For $\lambda = 0$ the system reduces to
\begin{align}
     \Div(\ro \u) = &\:  0,\label{feq47}  \\
     - \Div \field S(\ro,\nabla\u)  +   \nabla  \pi(\ro,0)=&\: \ro\mathbf{f},  \label{feq48} \\
    -\Div\Bigl(     (\delta + \mathrm{e}^z) {\kappa}_0 M(1+\mathrm{e}^{3z}) \nabla z \Bigr) = &\:\field{S}(\ro,\nabla \u):\nabla\u   , \label{feq49}\\
- \Div \biggl( \Bigl(D_0 M \mathrm{e}^{w_k} +\epsilon \Bigr) \nabla w_k \biggr)+ \epsilon w_k   + \delta \mathrm{e}^{w_k} =&\:\dfrac{\delta}{n}  , \quad k=1,2,\ldots,n \label{feq50}
\end{align}
with boundary conditions
\begin{align}
-\Bigl(D_0 M \mathrm{e}^{w_k} +\epsilon \Bigr) \nabla w_k \cdot \en = \:& 0, \quad k=1,2,\ldots,n, \label{beq1}\\
-  (\delta + \mathrm{e}^z)(1+\mathrm{e}^{3z})   \kappa_0 M \nabla z \cdot \en  = \:&   L_0 M(1+\mathrm{e}^{3z}) (\mathrm{e}^{z}-\Theta) + \epsilon  z ,\label{beq2}\\
 \u\cdot\en = \:& 0, \label{beq3}\\
 \en \cdot \field{S}(\ro,\nabla \u) \cdot \boldsymbol{\tau}^l + f \u \cdot \boldsymbol{\tau}^l = \:& 0, \quad l=1,2. \label{beq4}
\end{align}

\subsection{Properties of $\mathcal{F}_0$}

We start with the study of the case $\lambda=0$. We first show existence and uniqueness of the solution to this problem (it is also the fixed point here), then we verify that for suitably chosen constants $E$ and $C_{\mathbf{f}}$ the solution belongs to the balls defined below.

 The existence and uniqueness of fixed point for equations \refx{feq47}--\refx{feq48} with boundary conditions \eqref{beq3}--\eqref{beq4} follows from our previous works. For the reader's convenience we recall the corresponding Lemma 3.1 from \cite{AMP18}.  We use below the present notation.
\begin{lem}\label{Lem1}
Let $\Omega \in C^2$ be a bounded domain in $\field{R}^3$ which is not axially symmetric. Let $3<p<\infty$. Then there exist constants $C_\ef,\: E$ and $M_0>C_\ef$ dependent only on the given data, such that for any $M$ larger than $M_0$  problem \refx{feq47}--\refx{feq48} with boundary conditions \eqref{beq3}--\eqref{beq4}
admits a unique solution $(r,\u)$ in the class $M_r\times M_\u \cap W^{2,p}(\Omega;\mathbb{R}^{3})$ where
\begin{align*}
M_{\u}=&\Bigl\{ \vektor{f}\in  W^{1,\infty}(\Omega;\mathbb{R}^{3}), \:\vektor{f}\cdot\en={0} \text{ on }\partial \Omega, \Bigr.\nonumber\\ \Bigl.&\qqquad\qqquad \bigl\|{\nabla\vektor{f}}\bigr\|_2 \leq E, \: \|{\nabla \vektor{f}}\|_\infty+\norm{ \vektor{f}}_\infty   \leq C_\ef \Bigr\}, \\
M_r=& \biggl\{ f\in W^{1,p}(\Omega),\int_\Omega{f}\,\dx=0 , \: M^{\gamma-2}(\norm{f}_\infty+\norm{\nabla f}_p)\leq C_\ef \biggr\}.
\end{align*}
\end{lem}

In addition to the above defined subsets of function spaces we further define
%establish (note that the sets are connected with the formulation in logarithms of the temperature and mass fractions)
 \begin{align*}
M_\theta=&\biggl\{ f\in W^{1,p}(\Omega),\: \|f\|_{9}^{\frac 32}  +  \|\nabla |f^\frac 32| \|_{2} \leq E,\: \|f\|_{1,p}\leq C_\ef \biggr\},\\
M_{\vek{Y}}=&\biggl\{ \vek{f}\in W^{1,p}(\Omega;\field{R}^n),\: \| \vec f \|_{1,2} \leq E,\: \| \nabla \vec f \|_{p}\leq C_\ef ,\:k = 1,\ldots, n \biggr\},
\end{align*}
where  $C_\ef$ and $E$ depend only on the given data $\Omega,$ $p$, $\ef$, $\Theta$.

We  now look for existence of solutions to \eqref{feq49} and \eqref{feq50} with the corresponding boundary conditions (or, equivalently, equations \eqref{eq49} and \eqref{eq50} with the corresponding boundary conditions and $\lambda =0$). Note that these solutions (together with $\u$) are in fact the fixed points of $\mathcal F_0$.

We first consider \eqref{feq50} with the corresponding boundary conditions \eqref{beq1}. We apply the method of the Kirchhoff transform. We fix $k \in \{1,2,\dots, n\}$ and define
$$
h(x) = D_0M \mathrm{e}^x +\epsilon, \qquad H(w_k) = \int_0^{w_k} h(x)\, \mathrm{d}x
$$
and set $W_k = H(w_k)$. Then \eqref{feq50} with the boundary conditions \eqref{beq1} transforms into
\begin{equation} \label{KT}
\begin{aligned}
-\Delta W_k + \epsilon H^{-1} (W_k) + \delta \mathrm{e}^{H^{-1}(W_k)} &= \frac{\delta}{n} \\
\frac{\partial W_k}{\partial \mathbf{n}} &= 0.
\end{aligned}
\end{equation}
Note that $H(y) \sim \varepsilon y$ for $y\ll -1$, $H(y) \sim \epsilon y+1$ for $y$ close to $0$ and $H(y) \sim \mathrm{e}^y$ for $y \gg 1$ and note that $H(\cdot)$ is increasing in $\R$. Therefore also $H^{-1}$ behaves linearly in $\R^-$ and near $0$ while $H^{-1}(y) \sim \log y$ for $y \gg 1$ and $H^{-1}(\cdot)$ is also increasing in $\R$.

Due to the properties of $H$ and $H^{-1}$ there exists unique solution to \eqref{KT} in $W^{1,2}(\Omega)$. Moreover, it is also not difficult to verify that $W_k\in W^{2,r}(\Omega)$ for any $1\leq r<\infty$. From here, by a simple bootstrapping argument, we easily obtain that that also $w_k$ and $\mathrm{e}^{w_k} \in W^{2,r}(\Omega)$.

In a very similar way we also get that $z$ and $\mathrm{e}^z$, where $z$ solves \eqref{feq49} with the boundary condition \eqref{beq2}, belong to $W^{2,r}(\Omega)$ for any $1\leq r<\infty$. Hence we see that the operator $\mathcal{F}_0$ is in fact a compact operator from $X$ to $X$. The continuity of the operator is trivial.

It remains to show that the constructed solution (and thus also the fixed point) lies within our class $M_\theta$, and $M_{\vek{Y}}$ respectively. This is very simple. As explained above,  the functions $z$ and $w_k$ are sufficiently regular. Moreover, we may rewrite the equations for $\theta:={\mbox e}^z$ and $Y_k:= {\rm e}^{w_k}$. Then it is straightforward to get the estimates for $w_k$, first by testing the equation by $w_k$ and then using the elliptic regularity, similarly as in the previous section. To conclude, note that we can obtain the entropy end the total energy identity in a simplified form. To do so, we test the thermal equation by $\mathrm{e}^{-z}$ which gives us
\begin{multline*}
\inte{ \kappa_0 M \dfrac{\delta + \mathrm{e}^z}{\mathrm{e}^z}(1+\mathrm{e}^{3z}) \abs{\nabla z}^2 }+\inte{ M \mathrm{e}^{-z}\abs{\nabla \u}^2 } \\
+ L_0 M \inth{ (1+\mathrm{e}^{3z}) \mathrm{e}^{-z}\Theta } \leq L_0 M \inth{ (1+\mathrm{e}^{3z}) }
\end{multline*}
and denoting  $\theta:={\mbox e}^z$ we get the entropy identity. Next, simply integrating the equation and summing it with the momentum equation tested by the velocity yields the total energy balance
$$ L_0 M\inth{ (1+\theta^3)( \theta - \Theta)  } + \delta \inth {\log \theta}   = \inte{ \ro \mathbf{f}\cdot \u }.$$
Since we control the integral of $\frac{1}{\theta}$ over the boundary, there is no problem to estimate the additional term and we may proceed exactly as in the previous section to show the estimates in lower order norms (estimates by $E$).
To finish, we use the elliptic regularity to estimate also the $L^p$-norm of the temperature gradient. To do so, it is enough to estimate $\frac 1M \|\field{S}(\ro,\nabla \u):\nabla \u\|_{\frac {3p}{3+p}}$. It reads
$$ \frac 1M \norm{\field{S}:\nabla\u}_{\frac{3p}{3+p}} \leq C\norm{\nabla\u}_{\frac{6p}{3+p}}^2  \leq  C\norm{\nabla\u}_{2}^{\frac{6p-6}{5p-6}} \norm{\u}_{2,p}^{\frac{4p-6}{5p-6}}  \leq  C E^{\frac{6p-6}{5p-6}} C_\ef^{\frac{4p-6}{5p-6}}.$$
Therefore, after possible enlargement of constant $C_\ef$, we obtain $z \in M_\theta$. To show that $\vec{Y}\in M_{\vec{Y}}$, where $Y_k = \mathrm{e}^{w_k}$, is even simpler.

\subsection{Properties of $\mathcal{F}_\lambda$: well posedness, continuity and compactness}

Let us show that the solution operator $\mathcal{F}_\lambda$ is for $\lambda \in (0,1]$ well-defined, continuous and compact from $X$ to $X$. Furthermore, we also need that the fixed points are bounded in this space, but this will be shown in the next subsection, based on estimates of the fixed points in
$$M_\u \times M_\theta \times M_{\vec Y}.$$
We start with problem \eqref{eq7}--\eqref{eq8} with the corresponding boundary conditions \refx{uen}--\refx{slip}. We may again recall the result from \cite[Lemma 3.1]{AMP18}

 \begin{lem}\label{Lem2}
Let $\Omega\subset \field{R}^3 $ be a $C^2$ bounded domain which is not axially symmetric. Let $3<p<\infty$.  Then
there exist constants $C_\ef,\: E$ and $M_0>C_\ef$ dependent only on the given data, such that for any $M$ larger than $M_0$ and for any $\overline{\u}\in M_{\u}$, and $\overline{z}, \mathrm{e}^{\overline{z}}\in W^{1,\infty}(\Omega)$, $\lambda\in (0,1]$, problem \refx{eq7}--\refx{eq8} with boundary conditions \refx{uen}--\refx{slip}
admits a unique solution $(r,\u)$ in the class $M_r\times M_\u \cap W^{2,p}(\Omega;\mathbb{R}^{3})$.
\end{lem}

Note that it is important here that we may estimate the "temperature" in $M_\theta$ which is a consequence of the regularization of the equation for the temperature by the presence of the function $g(\cdot)$.

The unique solvability of problem \refx{eq9}--\refx{eq0} with the corresponding boundary conditions follows similarly as in the previous subsection. We only solve first the equation for $w_k$ and then put into $\mathbf{F}_{k\lambda}$ the found $\vec{w}$ and solve the equation for $z$.

 The regularity of the solutions to the elliptic equation yields the compactness of the mapping as the solution lies in $W^{2,p}(\Omega)$ which is compactly embedded into $W^{1,\infty}(\Omega)$.
% as well as to $W^{1,p}(\Omega).$

It remains to verify the continuity of the family of operators $\mathcal{F}_\lambda$, both in $\lambda$ and, for fixed $\lambda \in [0,1]$, from $X$ to $X$.
First, let us show the continuity in $\lambda$, more precisely

 \begin{lem} \label{lem4}
The mapping $\mathcal{F}_\lambda$: $(\overline{\u},\overline{\theta}, \overline{Y_k}) \mapsto ({\u},{\theta}, {Y}_k)$ is continuous in $\lambda$ as the operator from $[0,1]$ to $X$.
 \end{lem}

\begin{proof}
Let us consider $\lambda_1,\:\lambda_2$ and the corresponding solutions ${\u}_i,{\theta}_i, {Y}_{k,i},\:i=1,2.$
\begin{align}
     \Div(\ro_i \u_i) = &\:  0  \label{lam1} \\
     \lambda_i \ro_i \overline{\u} \cdot\nabla   \u_i - \Div \field S(\ro_i,\u_i)  +   \nabla  \pi\Bigl(\ro_i,\lambda_i\frac{\overline{\theta}}{g(\|\overline\theta\|_{1,p})}\Bigr)=&\: \ro_i\mathbf{f}, \label{lam2} \\
    -\Div\Bigl( \frac{\delta + \theta_i}{\theta_i} \overline{\kappa}_{\lambda,i} \nabla \theta_i\Bigr) =\field S(\ro_i,\u_i):\nabla\u_i  \nonumber \\
		-\lambda_i \ro_i  \frac{\overline{\theta}}{g(\|\overline\theta\|_{1,p})}\Div\u_i - &\: \lambda_i \ro_i \u_i \cdot\nabla \widetilde{e}_m(\ro_i,\overline{\theta},\overline{Y_k}) - \lambda_i \Div\bigl(\sumk\overline{\theta} c_{pk} {\mathbf{F}}_{k\lambda, i}  \bigr)   \label{lam3}\\
 \lambda_i \frac{\ro_i \overline{\u} \cdot\nabla   \overline{Y_k}}{g(\|\overline\theta\|_{1,p})} + \Div \bigl(\mathbf{F}_{k\lambda, i}  \bigr)  =&\: \lambda_i \ro_i \omega_k\bigl(\overline{Y_k},\overline{\theta}\bigr) - \epsilon \log Y_{k,i} -
  \delta  \overline{Y_k}  +  \dfrac{\delta}{n}. \label{lam4}
\end{align}

Let us take \refx{lam1}$_1$--\refx{lam1}$_2$, multiply by $\gamma M^{\gamma-2}(r_1-r_2)$ and integrate over $\Omega$,  this yields after some integration by parts
\begin{multline}
\inte{\gamma M^{\gamma-1} (\u_1-\u_2)\cdot\nabla (r_1-r_2) }\\ = \gamma M^{\gamma-2} \inte{  \frac{\abs{r_1-r_2}^2}{2}\Div \u_1 + (r_1-r_2) (\u_1-\u_2)\cdot\nabla r_2  +   r_2(r_1-r_2) \Div(\u_1-\u_2)}. \label{help} \end{multline}

Let us further test the difference of equations \refx{lam2}$_1$--\refx{lam2}$_2$ by the difference $\u_1-\u_2$ to obtain the following terms:
$$ (\lambda_1-\lambda_2)\inte{\ro_1\overline{\u}\cdot\nabla \u_1 (\u_1-\u_2)}\leq  (\lambda_1-\lambda_2) \|\ro_1 \|_\infty \|\overline{\u}\|_\infty \|\nabla \u_1\|_2 \|\u_1-\u_2\|_{1,2}$$
$$  \lambda_2\inte{(\ro_1-\ro_2)\overline{\u}\cdot\nabla \u_1 (\u_1-\u_2)}\leq \lambda_2 \|r_1-r_2\|_2 \|\overline{\u}\|_\infty \|\nabla \u_1\|_2  \|\u_1-\u_2\|_{1,2}$$
$$  \lambda_2\inte{\ro_2\overline{\u}\cdot\nabla (\u_1-\u_2) (\u_1-\u_2)}  = - \lambda_2\inte{\bigl(\ro_2 \Div(\overline{\u}) + \overline{\u} \cdot \nabla r_2) \dfrac{ |\u_1-\u_2|^2}{2} } $$
$$ \inte{ (\ro_1-\ro_2) \nabla \u_1 : \nabla  (\u_1-\u_2)} \leq \|r_1-r_2\|_2 \| \nabla \u_1  \|_\infty  \| \u_1-\u_2 \|_{1,2}    $$
$$ \inth{ \ro_2 | \u_1-\u_2|^2 } +  \inte{ \ro_2 |\nabla (\u_1-\u_2)|^2} \geq \dfrac{M}{2}  \| \u_1-\u_2 \|_{1,2}^2    $$
 \begin{multline*}
 \inte{ (\ro_1^\gamma - \ro_2^\gamma) \Div(\u_1-\u_2)  } \\
 \leq \gamma M^{\gamma-2}\|r_1-r_2\|_2 \Bigl(\|r_1-r_2\|_2 \|\Div \u_1 \|_\infty+\|\nabla r_2\|_p \|\u_1-\u_2\|_{1,2}    + \|r_2\|_\infty \|\Div(\u_1-\u_2)\|_2  \Bigr)\end{multline*}
 $$(\lambda_1-\lambda_2)\inte{\ro_1\frac{\overline{\theta}}{g(\|\overline\theta\|_{1,p})} \Div(\u_1-\u_2)}\leq  (\lambda_1-\lambda_2) \|\ro_1\|_\infty \|\overline{\theta}\|_2 \|\u_1-\u_2\|_{1,2} $$
$$\lambda_2\inte{(r_1-r_2)\frac{\overline{\theta}}{g(\|\overline\theta\|_{1,p})} \Div(\u_1-\u_2)} \leq  \lambda_2 \| r_1 - r_2 \|_2 \|\overline{\theta}\|_\infty \|\u_1-\u_2\|_{1,2}  $$
$$  \inte{(\ro_1-\ro_2) \ef \cdot  (\u_1-\u_2)} \leq  \|r_1-r_2\|_2 \|\ef\|_3  \| \u_1-\u_2 \|_{1,2},$$
where for the difference of the pressures we used \refx{help}. This can be combined to get
$$ M \| \u_1-\u_2 \|_{1,2}^2  \leq C(M)(\lambda_1-\lambda_2)^2  +   M^{\gamma-2}\|r_1-r_2\|_2^2 \|\Div \u_1 \|_\infty + C (M^{2\gamma-5} + M^{-1})\| r_1-r_2\|_2^2. $$
The second term on the right-hand side can be estimated by means of the Bogovskii estimates, namely
we test the difference of momentum equations by $\boldsymbol{\Phi} = \mathcal{B}[r_1-r_2]$ yielding estimate for
$ \gamma M^{\gamma-1} \|r_1-r_2\|_2^2 .$

We get the following terms:
$$ \inte{(\lambda_1-\lambda_2) \rho_1  \overline{\u} \cdot\nabla \u_1 \cdot \boldsymbol{\Phi} } \leq  |\lambda_1-\lambda_2| M \|\overline{\u}\|_{3} \|\nabla \u_1\|_2 \|r_1-r_2 \|_2  $$
$$ \inte{\lambda_2 (r_1-r_2)  \overline{\u} \cdot\nabla \u_1\cdot \boldsymbol{\Phi} } \leq  \|r_1-r_2\|_2^2  \|\overline{\u}\|_{\infty} \|\nabla \u_1\|_3$$
$$ \inte{\lambda_2 \rho_2  \overline{\u} \cdot\nabla (\u_1-\u_2)\cdot \boldsymbol{\Phi} } \leq M \|\overline{\u}\|_3 \|\nabla (\u_1-\u_2)\|_2 \|r_1-r_2\|_2 $$
$$ \inte{ (r_1-r_2) \nabla \u_1: \nabla\boldsymbol {\Phi}} \leq \|r_1-r_2\|_2^2  \|\nabla \u_1\|_{\infty} $$
$$ \inte{ \rho_2 \nabla (\u_1 -\u_2) :\nabla\boldsymbol {\Phi}}\leq  M \|\nabla (\u_1 -\u_2)\|_2  \|r_1-r_2\|_2$$
$$\inte{ \gamma M^{\gamma-1}(r_1-r_2)^2 + (r_1-r_2)^2\lambda_1 \overline{\theta}} \geq  M^{\gamma-1} \|r_1-r_2\|_2^2 $$
$$\inte{\rho_2 (\lambda_1-\lambda_2) \frac{\overline{\theta}}{g(\|\overline\theta\|_{1,p})}(r_1-r_2)}  \leq M |\lambda_1-\lambda_2| \| \overline{\theta}\|_2 \|r_1-r_2\|_2 $$
yielding by virtue of Young's inequality and the fact that $C_{\ef} \ll M^{\gamma-1} $
$$   M^{\gamma-1} \|r_1-r_2\|_2^2 \leq   C M^{3-\gamma} \Bigl( (\lambda_1-\lambda_2)^2 + \|\nabla (\u_1-\u_2)\|_2^2 \Bigr) .   $$
Therefore, we get
$$ M\| \u_1-\u_2 \|_{1,2}^2  \leq C(M)(\lambda_1-\lambda_2)^2 +  C \|\u_1-\u_2\|_{1,2}^2 \|\Div \u_1 \|_\infty + C (M^{-1} + M^{3-2\gamma
})\| \u_1-\u_2 \|_{1,2}^2 $$
and due to choice of $M$ finally
$$ \| \u_1-\u_2 \|_{1,2}^2  +   \|r_1-r_2\|_2^2 \leq C(M)(\lambda_1-\lambda_2)^2.  $$
Since the velocity and the density belong to $W^{2,p}(\Omega;\R^3)\times W^{1,p}(\Omega)$, by interpolation, we get that
$$
\|\u_1-\u_2\|_{1,\infty} \leq C|\lambda_1-\lambda_2|^\alpha
$$
for some positive $\alpha$. Next, in a straightforward way, combining energy method and elliptic regularity, we first get from \eqref{lam4}
$$
\sum_{k=1}^n\|\log Y_{k,1}-\log Y_{k,2}\|_{1,\infty}+ \|\vec{Y}_1-\vec{Y}_2\|_{1,\infty} + \|\vec{Y}_1-\vec{Y}_2\|_{2,q}  \leq C|\lambda_1-\lambda_2|^\alpha
$$
for some $\alpha >0$ and any $q<\infty$,
and then, using particularly the last estimate (for the second derivatives) in \eqref{lam3}, we also get
$$
\|\log \theta_1-\log \theta_2\|_{1,\infty}+ \|\theta_1-\theta_2\|_{1,\infty}   \leq C|\lambda_1-\lambda_2|^\alpha.
$$
This completes the proof of this lemma.
\end{proof}

To finish this subsection, we have to show the continuity of the operator $\mathcal{F}_\lambda$ as the operator from $X$ to $X$ for arbitrary $\lambda \in [0,1]$. To do so, we can combine results from \cite{AMP18} (in particular, Lemma 3.4 and comments below it) with standard elliptic regularity for the equations \eqref{eq49} and \eqref{eq50} with the corresponding boundary conditions to show that
 \begin{lem} \label{lem5}
The mapping $\mathcal{F}_\lambda$: $(\overline{\u},\overline{\theta}, \overline{Y_k}) \mapsto ({\u},{\theta}, {Y}_k)$ is continuous from $X$ to $X$ for any $\lambda \in [0,1]$.
 \end{lem}

\subsection{Boundedness of fixed points of $\mathcal{F}_\lambda$}

To conclude, we need to show that the possible fixed points of $\mathcal{F_{\lambda}(.)}$ are bounded in $X$. Our system reads (we {denote}  $\theta = \mathrm{e}^z , \: Y_k = \mathrm{e}^{w_k})$
 \begin{align}
     \Div(\ro \u) = &\:  0, \label{CEl} \\
\lambda \ro \u \cdot\nabla   \u - \Div \field S(\ro,\u)  +   \nabla  \pi\big(\ro,\lambda\frac{\theta}{g(\|\theta\|_{1,p})}\big)=&\: \ro\mathbf{f},   \label{MEl} \\
   -\Div\bigl( \mathcal{K} \nabla \theta\bigr) =\field{S}:\nabla\u -  \lambda \widetilde{\pi}_m(\ro,\theta) \Div\u - &\: \lambda \ro \u \cdot\nabla \widetilde{e}_m(\ro,\theta,Y_k) - \lambda \Div\Bigl(\sumk\theta c_{pk} \mathbf{J}_k  \Bigr)   ,\label{EEl}\\
\lambda \frac{\ro \u \cdot\nabla   {Y_k}}{g(\|\theta\|_{1,p})} + \Div  \mathbf{J}_k   =& \:\lambda \ro \omega_k-  \epsilon \log Y_k -  \delta  Y_k  +   \dfrac{\delta}{n}, \quad k=1,\ldots,n \label{Ykl}
\end{align}
with
$$\mathbf{J}_k = - \Bigl( D_\lambda + \dfrac{\epsilon + \delta Y_k}{Y_k} \Bigr) \nabla Y_k,\qquad \mathcal{K} =  \frac{\delta + \theta}{\theta}  {\kappa}_\lambda, \qquad \widetilde{\pi}_m(\ro,\theta) = \ro \frac{\theta}{g(\|\theta\|_{1,p})}$$
and $\widetilde{e}_m$ as in \eqref{eq50.a},
and boundary conditions
\begin{align}
\mathbf{J}_{k} \cdot \en = \:& 0, \quad k=1,2,\dots, n, \label{BC1l}\\
- \mathcal{K} \nabla \theta \cdot \en  = \:&   L_\lambda(\ro,\theta)(\theta-\Theta),\label{BC2l}\\
 \u\cdot\en = \:& 0,\label{BC3l}\\
 \en \cdot \field{S}(\ro,\nabla \u) \cdot \boldsymbol{\tau}^k + f \u \cdot \boldsymbol{\tau}^k = \:& 0.\label{BC4l}
\end{align}

Indeed, we have to bound not only $\theta$ and $\vec{Y}$, but also $\log \theta$ and $\{\log Y_k\}_{k=1}^n$. As a by-product we get that $(r,\u,\theta,\vec Y)$ are bounded in $M_r$, $M_{\u}$, $M_{\theta}$ and $M_{\vec Y}$, respectively, independently of $\lambda$. In particular, it will imply that for $\lambda=1$ the function $g(\|\theta\|_{1,p})$ equals to 1, id est, the fixed point is in fact (after the limit passages $\epsilon$ and $\delta \to 0$) a solution to our original problem. Moreover, as for $\epsilon$ and $\lambda >0$ we do not know that $\sum_{k=1}^nY_k=1$, we also show now that $\sum_{k=1}^nY_k-1 = o(\delta)$ provided $\epsilon = \epsilon(\delta)$, suitable (for our purpose, $\epsilon = \delta^3$ is enough).

We will closely follow the a priori estimates.

            %%%%%%%%%%%%%%%%%%%%%%%%%%  %%%%%%%%%%%%%%%%%%%%%%%%%%%
%%%%%%%%%%%%%%%%%%%%%%%%%  ESTIMATES WITH LAMBDA  %%%%%%%%%%%%%%%%%%%%%%%%%%%%%%
            %%%%%%%%%%%%%%%%%%%%%%%%%%  %%%%%%%%%%%%%%%%%%%%%%%%%%%

First of all,  we  deduce from \refx{Ykl} a useful estimate of quantity $\sigma_Y = \sumk Y_k$ which we finally want to be equal 1. We sum the equations for $k=1,\ldots, n$
\begin{multline*}
\lambda \frac{\ro \u \cdot\nabla (\sigma_Y-1)}{g(\|\theta\|_{1,p})} - \Div \Bigl( D_{\lambda}(\ro,{\theta},{\vek{Y}}) \nabla( \sigma_Y-1) \Bigr) - \delta \Delta (\sigma_Y -1)  +   \delta  (\sigma_Y -1)  \\ = \lambda \ro \sumk \omega_k - \epsilon \sumk \log Y_k +  \epsilon \sumk \Delta \log Y_k  ,
\end{multline*}
and test it by $(\sigma_Y-1)$ to get (recall $\sumk \omega_k =0 $)
\begin{multline}\inte{\bigl(\delta + D_\lambda(\ro,{\theta},{\vek{Y}})\bigr) |\nabla (\sigma_Y-1) |^2} +  \delta \inte{ (\sigma_Y-1)^2 } \\  \leq C \epsilon \sumk \inte{\Bigl(|\sigma_Y-1| |\log Y_k| +| \nabla \log Y_k||  \nabla(\sigma_Y-1)|\Bigr)},
\end{multline}
yielding
 \begin{equation}
 (M+\delta) \norm{\nabla(\sigma_Y-1)}_2^2 +  \delta \norm{\sigma_Y-1}_2^2 \leq  C \frac{\epsilon^2}{\delta} \sumk\biggl( \norm{\log Y_k}_2^2+\norm{\nabla \log Y_k}_2^2\biggr). \label{sigmaY}\end{equation}
In particular, we have
 $$\norm{\sigma_Y-1}_{1,2}^2 \leq  C \frac{\epsilon^2}{\delta^2} \sumk \norm{\log Y_k}_{1,2}^2.$$
We will be able to perform the limits $\epsilon, \:\delta \to 0+$ simultaneously. However, we need to guarantee not only that $\sigma_Y-1$ is bounded, but also that it vanishes in the limit. Therefore, we set for the rest of the proof\footnote{Actually, any power strictly larger than 2 would work in the same manner.} $$\epsilon = \delta^3$$  to get
\begin{equation} \norm{\sigma_Y-1}_{2}^2 \leq  C \delta^4 \sumk  \norm{\log Y_k}_{1,2}^2.\label{sigY}\end{equation}
Moreover, we obtain (as $Y_k>0$)
  \begin{equation}
   \norm{Y_k}_2^2 \leq \norm{\sigma_Y}_2^2 \leq \norm{\sigma_Y-1}_2^2 + \norm{1}_2^2 \leq  C \Bigl(1+ \delta^4  \sumk  \norm{\log Y_k}_{1,2}^2 \Bigr). \label{Yk2}
  \end{equation}
Observe also that we have
\begin{equation}
\sumk \mathbf{J}_k  = (D_\lambda + \delta) \nabla\sigma_Y + \delta^3 \sumk \nabla(\log Y_k),
\end{equation}
whence due to \refx{sigmaY}
\begin{equation}
\Bigl\|\sumk \mathbf{J}_k\Bigr\|_{2}  \leq M \norm{\nabla\sigma_Y}_2 + \delta^3 \sumk \norm{ \nabla(\log Y_k)}_2 \leq C  \sqrt{M}\delta^{5/2}  \sumk\norm{ \log Y_k}_{1,2}. \label{sumJk}
\end{equation}
Next, we multiply equation \refx{EEl} by $\dfrac{1}{\theta}$, yielding
\begin{multline}
\dfrac{2\ro\abs{\field D(\u)}^2}{\theta} + \dfrac{ \mathcal{K}\abs{\nabla \theta}^2 }{\theta^2} - {\Div\Bigl(\dfrac{ -\mathcal{K}\nabla \theta}{\theta} + \dfrac{\lambda \sumk h_k \mathbf{J}_k  }{\theta}\Bigr)}= \lambda {\dfrac{ \sumk h_k \mathbf{J}_k \cdot \nabla\theta }{\theta^2}} +\lambda{ \frac{\ro \Div\u }{g(\|\theta\|_{1,p})} }\\
+ \lambda \Bigl(  \frac{\ro \sumk c_{vk}Y_k  \u \cdot \nabla \theta}{g(\|\theta\|_{1,p})\theta}+\frac{\ro  \u}{g(\|\theta\|_{1,p})} \cdot  \sumk c_{vk}\nabla Y_k   \Bigr) , \label{entrl}
\end{multline}
and equation \refx{Ykl} by $\dfrac{g_k}{\theta}$; we obtain, see \refx{hk}--\refx{gk},
\begin{equation}
\lambda \frac{\Div(\ro Y_k \u)}{g(\|\theta\|_{1,p})}\dfrac{h_k-\theta s_k}{\theta} +{\Div \mathbf{J}_k \dfrac{g_k}{\theta} } =\lambda {\dfrac{\ro \omega_k g_k}{\theta}}  -  \epsilon\dfrac{ g_k\log Y_k}{\theta} - \delta \dfrac{ Y_k  g_k }{\theta}+  \delta \dfrac{ g_k}{n\theta} .
\end{equation}
This identity can be rewritten to
\begin{multline}
\lambda \frac{\ro  \u \cdot c_{pk} \nabla Y_k  }{g(\|\theta\|_{1,p})} -\lambda \frac{\Div(\ro Y_k \u) s_k}{g(\|\theta\|_{1,p})} + {\Div \Bigl(\mathbf{J}_k \dfrac{g_k}{\theta}\Bigr) }  -{\mathbf{J}_k \cdot\nabla \Bigl( \dfrac{g_k}{\theta}\Bigr) }  \\
 = \lambda {\dfrac{\ro \omega_k g_k}{\theta}}  - (c_{pk} - c_{vk} \log \theta + \log \ro + \log Y_k) \Bigl(\delta \bigl(Y_k-\frac{1}{n}\bigr)+\epsilon\log Y_k \Bigr). \label{Fklgk}
\end{multline}
Further,
\begin{equation*}
\inte{\frac{\ro  \u \cdot c_{pk} \nabla Y_k}{g(\|\theta\|_{1,p})}  } =\inte{\frac{\ro  \u \cdot c_{vk} \nabla Y_k}{g(\|\theta\|_{1,p})}  },
\end{equation*}
$$ \lambda\sumk\frac{\Div(\ro  \u  Y_k)s_k}{g(\|\theta\|_{1,p})} =\lambda\frac{\Div(\ro s  \u ) }{g(\|\theta\|_{1,p})}  - \lambda \dfrac{ \sumk  c_{vk} Y_k \ro  \u\cdot\nabla\theta}{g(\|\theta\|_{1,p})\theta} +  \lambda \frac{\sigma_Y \u  \cdot  \nabla\ro}{g(\|\theta\|_{1,p})}  + \lambda\frac{\ro\u \cdot \nabla \sigma_Y}{g(\|\theta\|_{1,p})} , $$
and
\begin{equation*}
-\sumk{\mathbf{J}_k \cdot\nabla \Bigl( \dfrac{g_k}{\theta}\Bigr) }
 =   {\sumk \dfrac{\mathbf{J}_k \cdot   c_{vk} \nabla \theta}{\theta}  } - {\sumk \dfrac{\mathbf{J}_k  \cdot \nabla Y_k}{Y_k}  }  - \sumk\mathbf{J}_k  \cdot \dfrac{\nabla \ro}{\ro} .
\end{equation*}
Thus, integrated version of \refx{Fklgk} reads
\begin{multline}
 \inte{\Div\Bigl(\sumk \mathbf{J}_k  \dfrac{g_k}{\theta} - \lambda  \frac{\ro s  \u}{g(\|\theta\|_{1,p})}  \Bigr) } +  \inte{\sumk \dfrac{\mathbf{J}_k   c_{vk} \cdot \nabla \theta}{\theta}  }+ \sumk  \inte{ \Big(\epsilon  |\log Y_k|^2  - \dfrac{\mathbf{J}_k   \cdot \nabla Y_k}{Y_k}\Big)}  \\ + \sumk \inte{ \delta \Bigl(1+(Y_k-\frac{1}{n})\log Y_k\Bigr) }
 =  \inte{ \Bigl(n \delta + \lambda \frac{\ro \u \cdot \nabla \sigma_Y}{g(\|\theta\|_{1,p})} +\lambda  \frac{\sigma_Y \u  \cdot  \nabla\ro}{g(\|\theta\|_{1,p})} \Bigr) }  \\  - \lambda  \inte{\sumk \frac{\ro  \u c_{vk} \cdot\nabla Y_k}{g(\|\theta\|_{1,p})}  }- \lambda \inte{     \dfrac{ \sumk  c_{vk} Y_k \ro  \u\cdot\nabla\theta}{g(\|\theta\|_{1,p})\theta}} +\inte{\sumk\mathbf{J}_k  \cdot \dfrac{\nabla \ro}{\ro}}
  \\ +  \lambda \inte{\sumk\ro \omega_k \dfrac{g_k}{\theta}}- \inte{\sumk (c_{pk} - c_{vk} \log\theta + \log \ro)\Bigl(\delta \bigl(Y_k-\frac{1}{n}\bigr) + \epsilon\log Y_k\Bigr)} ,\label{Eq66}
\end{multline}
while integrating \refx{entrl} gives us
\begin{multline*}
\inte{\dfrac{2\ro\abs{\field D(\u)}^2}{\theta} + \dfrac{ \mathcal{K} \abs{\nabla \theta}^2 }{\theta^2}} - \inte{\Div\Bigl(\dfrac{ -\mathcal{K}\nabla \theta + \lambda \sumk h_k \mathbf{J}_k   }{\theta}\Bigr)}\\=  \lambda \inte{\dfrac{ \sumk h_k \mathbf{J}_k  \cdot \nabla\theta }{\theta^2}}  +\lambda \inte{ \frac{\ro \Div\u }{g(\|\theta\|_{1,p})}}+ \lambda \inte{\dfrac{\ro  \u \cdot \nabla \bigl(\theta \sumk c_{vk}Y_k\bigr)}{\theta g(\|\theta\|_{1,p})}}.
\end{multline*}
We can conclude adding \refx{Eq66} and using the boundary conditions
\begin{multline*}
\inte{\dfrac{2\ro\abs{\field D(\u)}^2}{\theta} + \dfrac{\mathcal{K}\abs{\nabla \theta}^2}{\theta^2}}  -  \inte{\biggl[\sumk \dfrac{\mathbf{J}_k  \cdot \nabla Y_k}{Y_k}  + \epsilon  |\log Y_k|^2  +\delta \Bigl(1+\bigl(Y_k-\frac{1}{n}\bigr)\log Y_k\Bigr)\biggr] } \\ +\inth {L_\lambda \dfrac{\Theta}{\theta} }    = \inth {L_\lambda}+  \inte{\sumk (c_{pk} - c_{vk} \log\theta + \log \ro)\Bigl(\delta \Bigl(Y_k-\frac{1}{n}\Bigr) + \epsilon\log Y_k\Bigr)} \\
+\lambda \inte{\Bigl( n\delta + \frac{\ro (1-\sigma_Y)}{g(\|\theta\|_{1,p})}\Div \u\Bigr) }+ \sumk \inte{ \biggl(\lambda \ro \omega_k \dfrac{g_k}{\theta}+\mathbf{J}_k  \cdot \dfrac{\nabla \ro}{\ro} +  (c_{pk}-c_{vk})\mathbf{J}_k  \cdot \dfrac{\nabla \theta}{\theta}\biggr)} ,
\end{multline*}
whence, recalling $\epsilon = \delta^3$
\begin{multline*}
\inte{\dfrac{2\ro\abs{\field D(\u)}^2}{\theta} + \dfrac{\mathcal{K}\abs{\nabla \theta}^2}{\theta^2}}  + \lambda \inte{\sumk \dfrac{(D+\delta)|\nabla Y_k|^2}{Y_k} }  +\inth {L_\lambda \dfrac{\Theta}{\theta} }  \\ +\lambda \inte{\sumk \delta^3|\nabla\log Y_k|^2 +  \delta^3 |\log Y_k|^2  +\delta \bigl(1+Y_k\log Y_k\bigr) }  \leq   \inth {L_\lambda}+\lambda \inte{ \frac{\ro (1-\sigma_Y)\Div \u}{g(\|\theta\|_{1,p})} }\\+ \inte{ \biggl( \sumk \Bigl((c_{pk} - c_{vk} \log\theta + \log \ro)\Bigl(\delta \Bigl(Y_k-\frac{1}{n}\Bigr) +  \delta^3\log Y_k\Bigr)+\mathbf{J}_k  \cdot \dfrac{\nabla \ro}{\ro} +\mathbf{J}_k  \cdot \dfrac{\nabla \theta}{\theta} \Bigr)\biggr)} .
\end{multline*}
Let us estimate the lower order terms on the right-hand side of the last inequality. We have using \refx{sigY} (recall that $g\geq 1$)
  \begin{multline}
 \lambda \inte{\frac{\ro (1-\sigma_Y)\Div \u}{g(\|\theta\|_{1,p})}   }\leq  C \lambda  M \norm{1-\sigma_Y}_2 \norm{\Div\u}_2
   \leq C \lambda M \delta^2  \sumk\norm{ \log Y_k}_{1,2}\norm{\Div\u}_2 \\ \leq C  \lambda  M^2 \delta  \norm{\Div\u}_2^2 + \lambda \frac{\delta^3}{5} \sumk \norm{ \log Y_k}_{1,2}^2.
  \end{multline}
Similarly for the second term we have
\begin{equation}
\lambda \delta\inte{ (c_{pk} - c_{vk} \log\theta +  \log\ro)(Y_k-\frac1n) } \leq C\lambda\delta  \bigl(1+\| Y_k\|_2\bigr) \bigl(1+ \| \log\theta\|_{2} + M \bigr) ,
\end{equation}
and in order to eliminate the term we use for sufficiently small $\delta$, see \refx{Yk2}
$$ \norm{Y_k}_2^2 \leq C \Bigl(1+ \delta^4  \sumk  \norm{\log Y_k}_{1,2}^2 \Bigr) .$$
Next (recall that due to the construction we know that $\ro \geq \frac{M}{2}$)
\begin{multline}
\lambda \inte{ \delta^{3}(c_{pk} - c_{vk} \log\theta + \log \ro)\log Y_k } \leq \lambda \frac{\delta^{3}}{5} \norm{\log Y_k}_2^2 + C \lambda \delta^{3} \Bigl(1+\inte{\bigl(\log^2\theta + \log^2\ro\bigr)}\Bigr)\\ \leq \lambda \frac{\delta^{3}}{5} \norm{\log Y_k}_2^2 + C \lambda \delta^{3} \Bigl(M+\inte{\log^2\theta }\Bigr).
\end{multline}
Similarly, due to \refx{sumJk} and the fact that $r\in M_r$ we obtain
\begin{equation}
\lambda \inte{\sumk\mathbf{J}_k  \cdot \dfrac{\nabla \ro}{\ro} }\leq C\lambda \frac{\norm{\nabla r}_2}{M} \sqrt{M} \delta^{5/2} \sumk \norm{\log Y_k}_{1,2}  \leq  \lambda \frac{C\delta^2 }{M}  \norm{\nabla r}_2^2 +  \lambda\sumk\frac{\delta^3}{5} \norm{\log Y_k}_{1,2}^2 ,
\end{equation}
\begin{multline}
\lambda \inte{\sumk\mathbf{J}_k  \cdot \dfrac{\nabla \theta}{\theta} } \leq C\lambda \norm{\nabla \log\theta}_2 \sqrt{M} \delta^{5/2} \sumk \norm{\log Y_k}_{1,2} \\ \leq C\lambda \delta^2 M\norm{\nabla \log\theta}_2^2  +\lambda \frac{\delta^{3}}{5} \sumk \norm{\log Y_k}_{1,2}^2.
\end{multline}
Therefore, due to structural assumptions concerning $\kappa_\lambda$ and $L_\lambda$ we get after dividing by $M$
\begin{multline}
 \inte{ \Bigl( \dfrac{\abs{\field D(\u)}^2}{\theta} + \abs{\nabla \log\theta}^2+\lambda\theta\abs{\nabla \theta}^2+\sumk \lambda \dfrac{  \abs{\nabla Y_k}^2}{Y_k} \Bigr) } + \inth {\Bigl(\dfrac{1}{\theta}+ \lambda \theta^{2}\Bigr) }   \\ +\frac{\lambda}{M} \inte{\sumk\delta^3|\nabla\log Y_k|^2 + \delta^3 |\log Y_k|^2  +\delta \bigl(1+Y_k\log Y_k\bigr) }
 \leq C  \inth {\bigl(\lambda\theta^3+1\bigr) } \\  +  C \lambda \Bigl(\delta+\frac{1}{M}\inte{\bigl(\delta\theta + \delta^3\log^2\theta\bigr)} + \delta^2 \norm{\nabla\log\theta}_2^2+ \delta M \norm{\Div\u}_2^2  +\frac{\delta^2}{M^2} \norm{\nabla r}_2^2 \Bigr)  ,
\end{multline}
where the integral over $\Omega$ on the right-hand side can be absorbed by the left-hand side for $\delta$ sufficiently small, the same holds for the term with $\nabla \log\theta$. The last two terms are actually small due to the choice of $M$ and the definition of $M_\u$ and $M_r$ respectively. Therefore, we can write simply
\begin{multline}
 \inte{ \Bigl( \dfrac{\abs{\field D(\u)}^2}{\theta} + \abs{\nabla \log\theta}^2+\lambda\theta\abs{\nabla \theta}^2+\sumk \lambda \dfrac{  \abs{\nabla Y_k}^2}{Y_k} \Bigr) } + \inth {\Bigl(\dfrac{1}{\theta}+ \lambda \theta^{2} \Bigr)} \\ +\frac{\lambda}{M} \inte{\sumk\delta^3|\nabla\log Y_k|^2 + \delta^3 |\log Y_k|^2  +\delta \bigl(1+Y_k\log Y_k\bigr) }
 \leq  C \Bigl(\lambda \norm{\theta}_{L^{3}(\partial\Omega)}^3 +1\Bigr)  .\label{entropl}
\end{multline}
In particular
$$ \lambda \norm{\theta}_{9}^3 \leq C \lambda  \norm{\theta^{3/2}}_{1,2}^2 \leq C\lambda \biggl(\norm{\nabla(\theta^{3/2})}_{2}^2  +  \inth{\theta^3} \biggr)  \leq C \Bigl(\lambda \norm{\theta}_{L^{3}(\partial\Omega)}^3 +1\Bigr).$$
Furthermore, we have the total energy balance
$$ \inth{ \big(L_\lambda(\ro,\theta)(\theta-\Theta) + f \abs{\u}^2\big) }  = \inte{ \ro\mathbf{f}\cdot\u}  $$
\begin{equation}
\begin{split}
 \inth{\big(\theta + \lambda\theta^{4} +  \abs{\u}^2\big) }  &\: \leq C \Bigl(\inte{ \mathbf{f}\cdot\u}  + \inth{(1+\lambda\theta^{3})\Theta} \Bigr)   \\
 &\: \leq C \Bigl(\norm{\mathbf{f}}_{6/5}\norm{\u}_6 +   1 \Bigr) . \label{totall}
 \end{split}
 \end{equation}
 Therefore,
 $$\lambda \norm{\theta}_{9}^3 \leq  C \Bigl(\norm{\u}_6^{3/4}+1 \Bigr) .$$
 Looking back to the momentum equation, and multiplying by the velocity field yields
\begin{multline*}
 M  \norm{\nabla\u}_2^2+M\|\u\|_{L^2(\partial \Omega)}^2 \leq C \inte{ \bigl(\lambda \frac{\nabla(\ro\theta)}{g(\|\theta\|_{1,p})}\cdot\u + \ro\ef\cdot\u \bigr)}\\ \leq C \norm{\ro}_\infty \bigl( \lambda \norm{\theta}_2\norm{\Div\u}_2+\norm{\ef}_{6/5}\norm{\u}_6\bigr)
 \leq C (\Theta,\ef)\Bigl(\norm{\u}_6^{1/4}\norm{\Div\u}_2+1\Bigr),
 \end{multline*}
 and by Young's inequality
\begin{equation}
  \norm{\u}_{1,2}^2 \leq C (\Theta,\ef).
\end{equation}
Therefore, there exists constant $E$ dependent only on given data and independent of $\lambda$ and $\delta$ such that
\begin{equation}
  \norm{\u}_{1,2}  + \lambda \norm{\theta}_{9}^3+ \lambda  \norm{\theta}_{1,2}+\|\vek{Y}\|_{1,2}  + \norm{\log\theta}_{1,2}^2 + \lambda  \sumk\Bigl(\delta^3\norm{\log Y_k}_{1,2}^2  + \delta \norm{Y_k\log Y_k}_1\Bigr) \leq E. \label{ODHAD}
\end{equation}
The estimate of gradient $\vek{Y}$  independent of $\lambda$ comes simply from equation \refx{Ykl} tested by $Y_k$ and estimate \refx{Yk2}.
Furthermore, by \refx{sigmaY} we have
\begin{equation}
 \lambda \norm{\sigma_Y-1}_{1,2}^2 \leq \delta C E. \label{sigma}
 \end{equation}
The rest of the estimates closely follows the a priori approach. We obtain (simply considering $\lambda\theta$ instead of $\theta$)
\begin{equation}
 \norm{\u}_{2,p} + M^{\gamma-2}\norm{r}_{1,p} \leq C\bigl(\norm{\ef}_p + \lambda \norm{\nabla \theta}_p\bigr).
\end{equation}
Examining equation \refx{EEl} we get
\begin{equation*}
\begin{split}
M\norm{\theta}_{1,p} &\:\leq C \Bigl( M\norm{ \theta- \Theta }_{W^{-1/p,p}(\partial\Omega)} +  \norm{\field{S}:\nabla\u}_{\frac{3p}{3+p}} + \lambda\norm{ \widetilde{\pi}_m(\ro,\theta) \Div\u+\ro\u \cdot\nabla \widetilde{e}_m(\ro,\theta,\vec{Y}_k)}_{\frac{3p}{3+p}} \Bigr. \\ &\qquad\qquad\qquad\qquad\qquad\qquad\qquad\qquad\qquad\qquad\qquad\qquad\qquad\qquad \Bigl.+  \lambda \Bigl\|\sumk\theta c_{pk} \mathbf{J}_k\Bigr\|_p  \Bigr) \\
\norm{\theta}_{1,p} &\:\leq C \biggl( \norm{ \theta- \Theta }_{W^{-1/p,p}(\partial\Omega)} +  \norm{\nabla\u}_{\frac{6p}{3+p}}^2 + \lambda \norm{ \theta \Div\u}_{\frac{3p}{3+p}} +\lambda \norm{ \u \cdot\nabla \theta}_{\frac{3p}{3+p}} \\ &\qquad\qquad\qquad\qquad\qquad\qquad\qquad\qquad\qquad\qquad +\lambda \sumk \Bigl (\norm{ \u  \theta\cdot\nabla Y_k }_{\frac{3p}{3+p}}+ \frac{\lambda}{M}\norm{ \theta \mathbf{J}_k}_p   \Bigr) \biggr).
\end{split}
\end{equation*}
Let us estimate the terms on the right-hand side of the preceding inequality. We need to carefully interpolate between estimates coming from the entropy balance and the higher order norms. Using the H\"older and the Gagliardo--Nirenberg interpolation inequalities, we obtain
\begin{equation}
\norm{\nabla\u}_{\frac{6p}{3+p}}^2 \leq C \norm{\nabla\u}_{2}^{\frac{6p-6}{5p-6}} \norm{\u}_{2,p}^{\frac{4p-6}{5p-6}} \leq C \norm{\u}_{2,p}^{\frac{4p-6}{5p-6}}
 \end{equation}
\begin{equation}
\lambda \norm{\theta \Div \u}_{\frac{3p}{3+p}} \leq C \lambda \norm{\theta}_{9} \norm{ \nabla \u}_{\frac{9p}{9+2p}} \leq C \lambda \norm{\theta}_{9} \norm{ \nabla \u}_{2}^{1-t} \norm{  \u}_{2,p}^{t}
\end{equation}
\begin{equation}
\lambda\norm{\u \nabla \theta}_{\frac{3p}{3+p}} \leq C \norm{\u}_{6} \lambda \norm{ \nabla \theta}_{\frac{6p}{6+p}} \leq C \norm{\u}_{1,2}  \lambda \norm{ \nabla \theta }_{2}^{1-t'} \norm{ \nabla\theta}_{p}^{t'}
\end{equation}
with $t<1$ and $t'<1$. As also $\frac{4p-6}{5p-6}<1$, these three terms can be put directly to the left-hand side  by means of the Young inequality.
Further, we will use
$$\norm{\u}_\infty\leq C \norm{\u}_6^{1-r}\norm{\u}_{2,p}^r   \qquad\qquad  \lambda \norm{\theta}_\infty\leq C \lambda \norm{\theta}_9^{1-s}\norm{\theta}_{1,p}^s ,$$
where $r = \dfrac{p}{5p-6}$ and $s = \dfrac{p}{4p-9}.$ Thus,
\begin{equation}
\lambda \norm{ \u  \theta\cdot\nabla Y_k }_{\frac{3p}{3+p}} \leq \norm{ \u}_\infty\lambda \norm{\theta}_{9}\norm{\nabla Y_k }_{\frac{9p}{9+2p}} \leq C \norm{\u}_{2,p}^r  \norm{\nabla Y_k }_{\frac{9p}{9+2p}}.
\end{equation}
For $3<p\leq\frac{18}{5}$ we have $\frac{9p}{9+2p} \leq 2$ and therefore, we can absorb the term by the left-hand side. Otherwise, we interpolate
$$\norm{\nabla Y_k }_{\frac{9p}{9+2p}}  \leq C\norm{\nabla Y_k }_{2}^{1-\varsigma} \norm{\nabla Y_k }_{p}^\varsigma$$
with $\varsigma = \dfrac{5p-18}{9p-18}.$
Similarly,
$$\lambda \norm{ \theta \mathbf{J}_k }_p  \leq \lambda \norm{\theta}_\infty \norm{   \mathbf{J}_k }_{p} \leq C   \norm{\theta}_{1,p}^s \norm{   \mathbf{J}_k }_p . $$
Summing up the previous estimates yields
\begin{equation}
\norm{\theta}_{1,p} \leq C(\ef,\Theta) \Bigl(1+ \norm{\nabla Y_k }_{p}^{\frac{5p-6}{4p-6}\frac{ 5p-18 } { 9p-18} }   + \frac1M \norm{    \mathbf{J}_k }_p^{   \frac{4p-9}{3p-9}} \Bigr).
\end{equation}
Further, note that as $Y_k>0$
\begin{equation} M\norm{\nabla Y_k}_p+\delta^3 \norm{\nabla( \log Y_k)}_p  \leq C \norm{\mathbf{J}_k }_p \label{YkJk}
\end{equation} independently of $\lambda$, and
$$  \frac{5p-6}{4p-6}\frac{ 5p-18 } { 9p-18} < \frac{4p-9}{3p-9} \text{ for any }p>3,$$ whence
\begin{equation}
\norm{\theta}_{1,p} \leq C(\ef,\Theta) \Bigl(1  + \frac1M \norm{  \mathbf{J}_k }_p^{   \frac{4p-9}{3p-9}} \Bigr). \label{thetaFkl}
\end{equation}
Therefore, it remains to bound $\mathbf{J}_k$.  Equation \refx{Ykl} reads
$$ \Div\mathbf{J}_k = \lambda\ro  \omega_k -\lambda\frac{\Div(\ro  \u \cdot Y_k)}{g(\|\theta\|_{1,p})} - \delta^3 \log Y_k - \delta  Y_k  + \dfrac{\delta}{n} ,$$
and thanks to the Bogovskii-type estimates we get
\begin{equation}
 \frac{1}{M} \norm{\mathbf{J}_k }_p\leq C\bigl( \norm{ \omega_k}_{p} + \norm{\u}_{2p} \norm{ Y_k}_{2p} + \delta^3 M^{-1}\norm{\log Y_k}_p  + \delta M^{-1}\norm{ Y_k}_p \bigr).
\end{equation}
Using  the following interpolations
$$ \norm{\u}_{2p} \leq \norm{\u}_{6}^{1-q} \norm{\u}_{2,p}^q  \text{ with } q = \dfrac{p-3}{5p-6} ,\qquad  \norm{Y_k}_{2p} \leq \norm{Y_k}_{6}^{1-\beta} \norm{Y_k}_{1,p}^\beta  \text{ with } \beta = \dfrac{p-3}{3p-6} ,$$
$$ \norm{\log Y_k}_{p} \leq \norm{\log Y_k}_{2}^{1-t''} \norm{\log Y_k}_{1,p}^{t''}   \qquad  \norm{Y_k}_{p} \leq \norm{ Y_k}_{2}^{1-t'''} \norm{ Y_k}_{1,p}^{t'''}\text{ with } t'', t'''\in(0,1) ,$$
 we obtain
\begin{equation}
\frac{1}{M} \norm{\mathbf{J}_k }_p\leq
C\bigl( 1+ \lambda \norm{\theta}_{1,p}^q    \norm{ Y_k}_{1,p}^\beta + \delta^3 M^{-1}\norm{\log Y_k}_{1,p}^{t''} + \delta M^{-1}  \norm{ Y_k}_{1,p}^{t'''} \bigr)
\end{equation}
and with \refx{YkJk} in hands together with the Young inequality
\begin{equation}
 \frac{1}{M} \norm{\mathbf{J}_k }_p\leq C\bigl( 1+ \norm{\theta}_{1,p}^{\frac{p-3}{5p-6}\frac{3p-6}{2p-3}} \bigr).
\end{equation}
This yields due to \refx{thetaFkl}
\begin{equation}
\norm{\theta}_{1,p} \leq C(\ef,\Theta,\omega) \Bigl(1  +  \norm{  \theta}_{1,p}^{   \frac{4p-9}{3p-9}\frac{p-3}{5p-6}\frac{3p-6}{2p-3}} \Bigr).
\end{equation}
However,
$$ \frac{4p-9}{3p-9}\frac{p-3}{5p-6}\frac{3p-6}{2p-3}<1,\text{ for all }p>3.$$
Therefore, we can conclude
\begin{equation}
 \norm{\u}_{2,p} + M^{\gamma-2}\norm{r}_{1,p} +\norm{\theta}_{1,p}+\|\vek{Y}\|_{1,p} \leq C(\ef,\Theta,\omega_k). \label{ODHAD2}
 \end{equation}
Furthermore, coming back to \eqref{Ykl} with estimate \eqref{ODHAD2} available, it is not difficult to verify using standard elliptic regularity that also $\|\vec{Y}\|_{2,p}$ is bounded uniformly with respect to $\lambda$, $\epsilon$ and $\delta$ and $\{\log Y_k\}_{k=1}^n$ is bounded in $W^{2,p}$ uniformly with respect to $\lambda$ (however, not to $\epsilon$ and $\delta$). Using this fact, we use \eqref{EEl} and similarly get the bound for $\theta$ in $W^{2,p}(\Omega)$ uniformly with respect to $\lambda$, $\epsilon$ and $\delta$ and of $\log \theta$ uniformly with respect to $\lambda$ (but not with respect to $\epsilon$ and $\delta$). This finishes the proof of boundedness of possible fixed points of $\mathcal{F}_\lambda$ uniformly with respect to $\lambda$ and therefore $\mathcal{F}_1$ possesses at least one fixed point in $X$. Moreover, the fixed point fulfils
\begin{equation}
 \norm{\u}_{2,p} + M^{\gamma-2}\norm{r}_{1,p} +\norm{\theta}_{1,p}+\|\vek{Y}\|_{1,p} \leq C(\ef,\Theta,\omega_k) \label{ODHAD3}
 \end{equation}
with the right-hand side independent of $M$, $\epsilon$ and $\delta$. Therefore, for a suitable choice of $C_0$ in the definition of the function $g(\cdot)$ (and the connected choice of the lower limit of $M$) we see that the function $g =1$. Finally, using the elliptic regularity, we can also easily that
\begin{equation}
\norm{\nabla^2 \theta}_{p}+\|\nabla^2  \vek{Y}\|_{p} \leq C(M,\ef,\Theta,\omega_k). \label{ODHAD4}
 \end{equation}

\subsection{Proof of the main theorem}

It remains to perform the limit passage as $\delta \to 0+$; recall that we set $\epsilon = \delta^3.$ The available bounds \refx{ODHAD}, \refx{ODHAD3} and \refx{ODHAD4} are sufficient to manage all terms, including the nonlinear ones. In particular from \refx{sigma} we achieve the condition $\sumk Y_k=1,$
 $$\delta^3 \norm{\log Y_k}_{1,2}\leq \delta^{3/2} E^{1/2},$$
thus $Y_k\geq 0$ and consequently also $Y_k\leq1.$  Considerations for the other terms are similar and we skip the details. This completes the proof of our main result.

\end{document}